\documentclass[a4paper,11pt]{amsart}
\title{Stable pairs on nodal $K3$ fibrations}
\date{}
\author{Amin Gholampour and Artan Sheshmani and Yukinobu Toda}

\usepackage{epstopdf}
\usepackage{mathtools}
\usepackage[cmtip,all]{xy}

\usepackage{mathrsfs}
\usepackage{amscd}
\usepackage{amsmath}
\usepackage{amssymb}
\usepackage{amsthm}
\usepackage{float}
\usepackage[dvips]{graphicx}
\usepackage{color}
\usepackage{tikz} 

\usetikzlibrary{matrix,arrows}

\usepackage{array}
\usepackage{amscd}
\usepackage[all]{xy}

\DeclareFontFamily{U}{rsfs}{%
\skewchar\font127}
\DeclareFontShape{U}{rsfs}{m}{n}{%
<-6>rsfs5<6-8.5>rsfs7<8.5->rsfs10}{}
\DeclareSymbolFont{rsfs}{U}{rsfs}{m}{n}
\DeclareSymbolFontAlphabet
{\mathrsfs}{rsfs}
\DeclareRobustCommand*\rsfs{%
\@fontswitch\relax\mathrsfs}

\theoremstyle{plain}
\newtheorem{thm}{Theorem}[section]
\newtheorem{theorem}{Theorem}
\newtheorem{prop}[thm]{Proposition}
\newtheorem{lem}[thm]{Lemma}

\newtheorem{defi}[thm]{Definition}
\newtheorem{rmk}[thm]{Remark}
\newtheorem{cor}[thm]{Corollary}

\newtheorem{step}{Step}

\newtheorem{prop-defi}[thm]{Proposition-Definition}
\newtheorem{thm-defi}[thm]{Theorem-Definition}
\newtheorem{lem-defi}[thm]{Lemma-Definition}

\newtheorem{conj}[thm]{Conjecture}
\newtheorem{exam}[thm]{Example}

\newdimen\argwidth
\def\db[#1\db]{
 \setbox0=\hbox{$#1$}\argwidth=\wd0
 \setbox0=\hbox{$\left[\box0\right]$}
  \advance\argwidth by -\wd0
 \left[\kern.3\argwidth\box0 \kern.3\argwidth\right]}

\newcommand{\X}{\mathfrak{X}}
\newcommand{\CC}{\mathbb{C}}
\renewcommand{\P}{\mathbb{P}}

\newcommand{\bP}{{\bf{P}}}

\newcommand{\C}{C}

\newcommand{\V}{\mathcal{V}}
\renewcommand{\O}{\mathcal{O}}
\newcommand{\Z}{\mathbb{Z}}
\newcommand{\Q}{\mathbb{Q}}
\newcommand{\F}{\mathcal{F}}

\newcommand{\K}{\mathcal{K}}

\newcommand{\GG}{\mathbb{G}}

\newcommand{\Extt}{\mathcal{E}xt}

\newcommand{\M}{\mathcal{M}}

\newcommand{\hilb}{\operatorname{Hilb}}

\newcommand{\Ib}{I^{\bullet}}
\newcommand{\Ibb}{\mathbb{I}^{\bullet}}
\newcommand{\kt}{\mathsf{K}3}

\newcommand{\aA}{\mathcal{A}}
\newcommand{\bB}{\mathcal{B}}
\newcommand{\cC}{\mathcal{C}}
\newcommand{\dD}{\mathcal{D}}
\newcommand{\eE}{\mathcal{E}}
\newcommand{\fF}{\mathcal{F}}
\newcommand{\gG}{\mathcal{G}}
\newcommand{\hH}{\mathcal{H}}

\newcommand{\kK}{\mathcal{K}}
\newcommand{\lL}{\mathcal{L}}
\newcommand{\mM}{\mathcal{M}}

\newcommand{\oO}{\mathcal{O}}
\newcommand{\pP}{\mathcal{P}}

\newcommand{\qQ}{\mathcal{Q}}

\newcommand{\sS}{\mathcal{S}}
\newcommand{\tT}{\mathcal{T}}
\newcommand{\uU}{\mathcal{U}}
\newcommand{\vV}{\mathcal{V}}

\newcommand{\Supp}{\mathop{\rm Supp}\nolimits}
\newcommand{\Hom}{\mathop{\rm Hom}\nolimits}

\newcommand{\dR}{\mathbf{R}}
\newcommand{\dL}{\mathbf{L}}

\newcommand{\Hilb}{\mathop{\rm Hilb}\nolimits}

\newcommand{\Pic}{\mathop{\rm Pic}\nolimits}

\newcommand{\tC}{\widetilde{C}}
\newcommand{\tX}{\widetilde{X}}

\newcommand{\tpi}{\widetilde{\pi}}

\newcommand{\id}{\textrm{id}}

\newcommand{\ch}{\mathop{\rm ch}\nolimits}

\newcommand{\td}{\mathop{\rm td}\nolimits}
\newcommand{\Ext}{\mathop{\rm Ext}\nolimits}

\newcommand{\rank}{\mathop{\rm rank}\nolimits}
\newcommand{\Coh}{\mathop{\rm Coh}\nolimits}

\newcommand{\Per}{\mathop{\rm Per}\nolimits}

\newcommand{\cneq}{\mathrel{\raise.095ex\hbox{:}\mkern-4.2mu=}}
\newcommand{\eqcn}{\mathrel{=\mkern-4.5mu\raise.095ex\hbox{:}}}

\newcommand{\Cok}{\mathop{\rm Cok}\nolimits}

\newcommand{\Aut}{\mathop{\rm Aut}\nolimits}

\newcommand{\DT}{\mathop{\rm DT}\nolimits}
\newcommand{\PT}{\mathop{\rm PT}\nolimits}

\newcommand{\Imm}{\mathop{\rm Im}\nolimits}
\newcommand{\imm}{\mathop{\rm im}\nolimits}

\newcommand{\Ker}{\mathop{\rm Ker}\nolimits}

\newcommand{\cl}{\mathop{\rm cl}\nolimits}

\begin{document}
\maketitle

\begin{abstract}
We study
Pandharipande-Thomas's stable pair theory
on $K3$ fibrations over curves with possibly nodal fibers. 
We describe stable pair invariants of the fiberwise irreducible
curve classes in terms of Kawai-Yoshioka's formula for the Euler characteristics of moduli spaces of stable pairs on $K3$ surfaces and Noether-Lefschetz numbers of the fibration. Moreover, we investigate the relation of these invariants with the perverse (non-commutative) stable pair invariants of the $K3$ fibration. In the case that the $K3$ fibration is a projective Calabi-Yau threefold, by means of wall-crossing techniques, we write the stable pair invariants 
in terms of the generalized Donaldson-Thomas invariants of 2-dimensional Gieseker semistable sheaves supported on the fibers.  
\end{abstract}

\setcounter{tocdepth}{1}
\tableofcontents


\section{Introduction}
\subsection{Overview}
The Fourier-Mukai transform is an 
equivalence of derived
categories of coherent sheaves 
of algebraic varieties, 
which was first studied
for dual pairs of abelian varieties in the seminal paper 
of Mukai~\cite{Mu1}.
When a Calabi-Yau 3-fold 
admits an elliptic fibration, 
its relative Fourier-Mukai transform
turned out to be an effective
tool 
to prove the 
correspondence
between one
dimensional sheaf theory and 
two dimensional sheaf theory~\cite{FM-mirror}, \cite{T-dual}. 
In particular, it provides 
a recipe to relate 
two kinds of invariants:
 \textit{Pandharipande-Thomas stable pair invariants}~\cite{a17}
which count curves in Calabi-Yau 3-folds, 
and \textit{Donaldson-Thomas invariants}
which count two dimensional 
semistable sheaves on them~\cite{a20}.  

The above correspondence is 
also important in string theory. 
The relative Fourier-Mukai transform 
is a mathematical
interpretation of relative T-duality in string theory, 
which was used by physicists
\cite{Klemm:2012sx, E-string, Black-T-dual}
to discuss the correspondence between the BPS theory of D2 branes wrapped on the base and the D4 branes which also wrap the elliptic fibers.
The PT stable pairs and the
two dimensional semistable sheaves are the
mathematical analogues of the D6-D2-D0 branes, 
and the D4 branes respectively,
and the relative T-duality
provides a recipe to relate 
to the BPS state counting of these D-brane
systems.


In a more general setting, without assuming that 
the threefold admits an elliptic fibration, we 
can still ask the question regarding the 
correspondence 
of PT stable pair theory and 
the DT theory of two dimensional 
semistable sheaves
(i.e. D6-D2-D0 BPS theory and D4 theory in string theory). 
This question shapes the main motivation behind the current articles. 
In~\cite{TodBg}, the third author 
derived the formula relating 
 the low degree terms of the 
generating series of the latter invariants with 
those of the former invariants, 
using the wall-crossing technique and some ideas 
from string theory~\cite{DM}. 
In \cite{G-S}, the first and the second authors studied the generating series of DT invariants counting two dimensional stable sheaves 
on $K3$ fibered threefolds with possibly nodal fibers, 
and proved their  
modularity property.  

The purpose of this article is to 
obtain an analogue of the results in~\cite{G-S} 
for PT stable pair theory
on $K3$ fibered threefolds. 
 Namely 
we investigate the modularity property of 
the stable pair
invariants on 
$K3$ fibrations. 
We first treat the case of 
a smooth $K3$ fibration, and 
relate the stable pair invariants 
with the generating series satisfying the 
modularity property. 
When the $K3$ fibration has possibly nodal fibers,
we relate their stable pair invariants with 
those on smooth $K3$ fibrations via conifold transition. 
In a more general situation, we 
also prove a formula 
relating stable pair invariants with the 
DT invariants counting two dimensional semistable sheaves, 
 where the latter invariants are expected to have the 
modular invariance property. 



\subsection{Main result I: smooth $K3$ fibration}
Let $X$ be a
smooth projective threefold over $\mathbb{C}$. 
For $\beta \in H_2(X, \Z)$ and $n\in \Z$, the moduli space of stable pairs $\pP_n(X,\beta)$ and the associated invariants were introduced by Pandharipande and Thomas \cite{a17}.  
The moduli space $\pP_n(X, \beta)$
 parametrizes pairs 
\begin{align}\label{two:term}
s \colon \O_{X} \to F
\end{align}
 where $F$ is a pure 1-dimensional sheaf on $X$ with $[F]=\beta$ and $\chi(F)=n$, and the cokernel of $s$ is 0-dimensional. 
They 
constructed a perfect obstruction theory on $\pP_n(X,\beta)$ (see Theorem \ref{PT})
by interpreting it as the moduli space of 
two term complexes (\ref{two:term}) in the derived category. 
When the virtual dimension of 
$\pP_n(X, \beta)$ is zero, the stable pair invariant 
$P_{n,\beta}$ is defined by taking the degree of the virtual cycle obtained from this obstruction theory. 

We assume that $X$ admits a morphism
\begin{align}\label{intro:pi}
\pi \colon X \to C
\end{align}
onto a smooth projective curve $C$, whose 
generic fiber is a smooth $K3$ surface. 
The morphism $\pi$ is called a \textit{K3 fibration}. 
We always take the 
curve class $\beta$ 
to be in the kernel of $\pi_*$ denoted by $H_2(X,\Z)^\pi$. The virtual dimension of $\pP_n(X, \beta)$ is zero for such curve classes\footnote{Note that $X$ is not required to be a Calabi-Yau 3-fold here.}.  
We define the following generating series:
\begin{align}\label{def:gen}
\PT(X)_{\beta}=\sum_{n\in \mathbb{Z}}
P_{n, \beta}q^n, \ 
\PT(X)=\sum_{\beta \in H_2(X, \mathbb{Z})^{\pi}}
\PT(X)_{\beta}t^{\beta}. 
\end{align} 

In Section~\ref{sec:PTNL}, we 
study the case when $\pi$ is a
smooth morphism. The majority of this section (Subsections 2.1-2.6) is devoted to the case where
$\beta$ is an irreducible  
curve class. In this case, we relate $P_{n, \beta}$ with 
the Euler characteristics of the moduli space of stable pairs on a 
nonsingular $K3$ surface $S$ and the Noether-Lefschetz numbers 
$NL^{\pi}_{h,\beta}$ of the fibration having modular properties \cite{a90, a116, a125, a117}. 
We denote by 
$\pP_{n}(\kt,h)$ the moduli space of stable pairs on a $K3$
surface containing an irreducible curve class $\gamma$
satisfying $\gamma^2=2h-2$.
It is known that the topological type
of $\pP_n(\kt,h)$ only depends on 
$(n, h)$ (cf.~\cite{MPT}). 
 \begin{theorem}\emph{(Theorem~\ref{thm:main formula})} \label{thm:PTNL}
Let $\pi \colon X\to C$ be a smooth $K3$ fibration, 
and $\beta \in H_2(X, \mathbb{Z})^{\pi}$ be an irreducible 
curve class. 
Then we have the following formula: 
\begin{align*}
\PT(X)_\beta&=\sum_{h=0}^{\infty}\sum_{n=1-h}^{\infty}(-1)^{n-1}\chi(\pP_{n}(\kt,h))\cdot NL^{\pi}_{h,\beta}\; q^{n}. 
\end{align*}
\end{theorem}
The proof of Theorem \ref{thm:PTNL} involves studying the restriction of the obstruction theory constructed in \cite{a17} to isolated and non-isolated components of $\pP_n(X,\beta)$. 
The Euler characteristics $\chi(\pP_{n}(\kt,h))$ in Theorem \ref{thm:PTNL} can be read off from Kawai-Yoshioka's formula having modular properties \cite[Theorem 5.80]{a79}: 
\begin{align} \label{equ:KW}
&
\sum_{h=0}^{\infty}\sum_{n=1-h}^{\infty}(-1)^{n-1}\chi(\pP_{n}(\kt,h))y^{n}q^{h}=\notag\\
&
-\left(\sqrt{-y}-\frac{1}{\sqrt{-y}}\right)^{-2}\prod_{n=1}^{\infty}\frac{1}{(1-q^{n})^{20}(1+yq^{n})^{2}(1+y^{-1}q^{n})^{2}}.
\end{align}

In Subsection \ref{sec:reducible}, we study a special analog of Theorem \ref{thm:PTNL} in which the class $\beta \in H_2(X, \mathbb{Z})^{\pi}$ is allowed to be reducible.

\subsection{Main result II: nodal $K3$ fibration}
We next study the case that the 
$K3$ fibration (\ref{intro:pi})
has a finite number of nodal fibers.
In this case, we call a morphism (\ref{intro:pi}) 
 a \textit{nodal K3 fibration}.  
For a nodal $K3$ fibration,  
we find a conifold transition formula, 
which relate the stable pair invariants on nodal $K3$ fibrations with 
those on smooth $K3$ fibrations. 
The following result is based on assuming some foundation of stable pair theory of algebraic spaces that has not been established yet (see Remarks \ref{rem:nonproj} and \ref{alg-space-degen}): 
\begin{theorem}\emph{(Theorem~\ref{thm:formula:relation})}
Let $\pi \colon X \to C$
be a nodal $K3$ fibration. 
Then there is a 
smooth $K3$ fibration\footnote{Here the situation is a little more 
general than (\ref{intro:pi}), and we allow $\widetilde{X}$ to 
be a proper algebraic space.} 
$\widetilde{\pi} \colon 
\widetilde{X} \to \widetilde{C}$
and morphisms
\begin{align*}
\widetilde{\epsilon} \colon 
\widetilde{X} \stackrel{h}{\to} X_0 \stackrel{\epsilon}{\to} X
\end{align*}
where $h$ is a small resolution and 
$\epsilon$ is a double 
cover, such that the 
following formula holds: 
\begin{align}\label{formula:relation}
\widetilde{\epsilon}_{\ast} \frac{\PT(\tX)}{\PT^{h}(\tX)}
= \PT(X)^2. 
\end{align}
Here $\widetilde{\epsilon}_{\ast}$
is the variable change $t^{\beta} \mapsto 
t^{\widetilde{\epsilon}_{\ast}\beta}$, 
and $\PT^h(\tX)$ is the subseries of $\PT(\tX)$
of stable pair invariants of curve classes $\widetilde{\beta}$
with $h_{\ast} \widetilde{\beta}=0$. 
\end{theorem}
Let us consider the invariant
$P_{n, \beta}$ with an irreducible curve class
$\beta$
on a nodal $K3$ fibration $\pi \colon X \to C$. 
The above result 
leads us to  
working with the stable pair invariants 
$P_{\tilde{n},\tilde{\beta}}$ 
for possibly reducible classes $\tilde{\beta}$ on 
a smooth $K3$ fibration $\widetilde{X} \to \widetilde{C}$
lifting $\beta$. In Subsection \ref{sec:noncomm}, we present an approach to compute the invariants $P_{\tilde{n},\tilde{\beta}}$ by studying the relationship between the invariants $P_{\tilde{n},\tilde{\beta}}$ and the invariants of the moduli spaces 
of perverse (non-commutative) stable pairs. 

\subsection{Main result III: product expansion formula}
Finally let $X$ be a smooth projective Calabi-Yau 3-fold, 
which admits a $K3$ fibration $\pi \colon X \to \mathbb{P}^1$
such that every fiber is an integral scheme. 
In Section \ref{sec:wall}, using the wall-crossing techniques mostly developed by the third author \cite{Tcurve2, Tcurve1, Tolim2, Tsurvey, TodK3, TodHall}, 
we express the generating series (\ref{def:gen}) in terms of the generalized Donaldson-Thomas invariants 
\begin{align}\label{intro:J}
J(r,\beta,n) \in \mathbb{Q}
\end{align}
which count semistable sheaves $E$ supported on the fibers
of $\pi$ satisfying 
\begin{align} \label{equ:chern} \ch(E)=(0, r[F], \beta, n)\in
H^0(X) \oplus H^2(X) \oplus H^4(X) \oplus H^6(X). 
\end{align} 
Here $[F]$ is the class of a fiber of $\pi$
in $X$. The invariants (\ref{intro:J})
are defined by the method of Joyce-Song \cite{JS}. These invariants were studied by the first and second authors in \cite{G-S} when there are no strictly semistable sheaves $F$ satisfying \eqref{equ:chern}. 
The modularity property of the invariants (\ref{intro:J})
holds in the situation of~\cite{G-S}, and we expect 
such a property in general. 
The following result gives a complete answer on 
describing the relationship between stable pair invariants and 
DT type invariants counting two dimensional sheaves: 
 \begin{theorem}\label{thm:wall}\emph{(Theorem~\ref{thm:WCF})}
We have the following formula
\begin{align*}\notag
\PT(X)=\prod_{r\ge 0, \beta>0, n\ge 0} &
\exp\left((-1)^{n-1}J(r, \beta, r+n)q^n t^{\beta} \right)^{n+2r} \\
\label{main:formula}
&\cdot \prod_{r> 0, \beta>0, n>0}
\exp\left((-1)^{n-1}J(r, \beta, r+n) q^{-n}t^{\beta} \right)^{n+2r}.
\end{align*}
\end{theorem}
The above 
result is proved by generalizing the argument in~\cite{TodK3}
showing a similar result when $X=S \times \mathbb{C}$
for a $K3$ surface $S$. 
Namely, we apply Joyce-Song's 
wall-crossing formula~\cite{JS}
in a certain space of weak stability conditions. 
As the arguments are the almost same as in~\cite{TodK3}
except a few modifications, 
we will just outline its proof in Section~\ref{sec:wall}.  

\subsection{Summary}
The main purpose of the current article is to find the connection between the vertices of the following triangle:

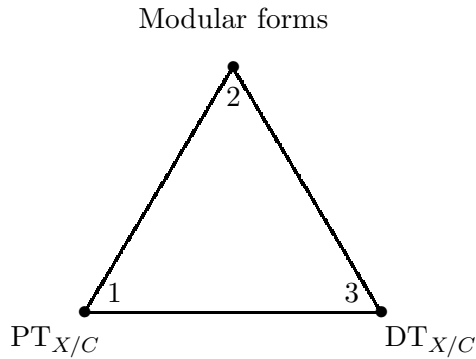
\ifx\JPicScale\undefined\def\JPicScale{1.3}\fi
\unitlength \JPicScale mm
\begin{picture}(100,40)(40,23)
\linethickness{0.3mm}
\multiput(70,30)(0.12,0.2){125}{\line(0,1){0.2}}
\linethickness{0.3mm}
\multiput(85,55)(0.12,-0.2){125}{\line(0,-1){0.2}}
\linethickness{0.3mm}
\put(70,30){\line(1,0){30}}
\put(85,55){\makebox(0,0)[cc]{$\bullet$}}

\put(70,30){\makebox(0,0)[cc]{$\bullet$}}

\put(100,30){\makebox(0,0)[cc]{$\bullet$}}

\put(67,27){\makebox(0,0)[cc]{$\text{PT}_{X/C}$}}

\put(105,27){\makebox(0,0)[cc]{$\text{DT}_{X/C}$}}

\put(85,60){\makebox(0,0)[cc]{Modular forms}}

\put(73,32){\makebox(0,0)[cc]{$1$}}

\put(85,52){\makebox(0,0)[cc]{$2$}}

\put(97,32){\makebox(0,0)[cc]{$3$}}

\end{picture}
Here by $\text{PT}_{X/ C}$ and $\text{DT}_{X/ C}$ we mean the fiberwise stable pair and Donald\-son-Thomas theories of the Calabi-Yau $K3$ fibration $X\to C$ discussed above. As mentioned earlier, the connection $1\leftrightarrow 3$ is motivated by $T$-duality and we provide a complete answer to it in 
Section~\ref{sec:wall}. The connection $3 \leftrightarrow 2$
 is motivated by $S$-duality, 
and the first two authors have provided partial results on it in \cite{G-S}. In Section~\ref{sec:PTNL} and Section~\ref{sec:nodalK3}
of this paper, we provide partial 
results on the connection between $1 \leftrightarrow 2$. More recently, Pandharipande-Thomas \cite{KKVPT} studied the connection $1\leftrightarrow 2$ in 
more generality by a different method (degeneration method) and for a different purpose. A refined (motivic) version of connection $1 \leftrightarrow 2$ is conjectured in \cite{KKP}. We hope the combination of the established bridges 
$1\leftrightarrow 2$ and $1\leftrightarrow 3$, sheds some lights on the connection $3\leftrightarrow 2$
 in full generality, which we hope we can work out in the future. 

\subsection{Acknowledgment}
We would like to thank Jan Manschot and Richard Thomas for helpful discussions. The first author was partially supported by NSF grant DMS-1406788.
The second author would like to thank MIT and the Institute for the Physics and Mathematics of the Universe (IPMU) for hospitality and providing the opportunity to  discuss about this project during his visits. Special thanks to Max Planck Institut f\"{u}r Mathematik for hospitality during the second author's stay in Bonn.  
The second and third authors are supported by World Premier 
International Research Center Initiative
(WPI initiative), MEXT, Japan. The third author
 is also supported by Grant-in Aid
for Scientific Research grant (No.~26287002)
from the Ministry of Education, Culture,
Sports, Science and Technology, Japan.

\subsection{Notation and convention}
In this paper, all the 
varieties or schemes are
defined over $\mathbb{C}$. 
For a morphism 
$\pi \colon X \to Y$
of schemes 
and coherent sheaves $E$, $F$ on $X$, 
we 
write the sheaf 
$\pi_{\ast} \hH om_{X}(E, F)$ as 
$\hH om_{\pi}(E, F)$. 
Its right derived functor is denoted by 
$\dR \hH om_{\pi}(E, F)$, which 
coincides with $\dR \pi_{\ast} \dR \hH om(E, F)$. 
If $\pi$ is smooth, we denote by 
$\omega_{\pi}$ the relative canonical 
line bundle on $X$. 
For a scheme $X$, we denote
by $\mathbb{L}_{X}^{\bullet}$ the cotangent complex of $X$. 

\newpage
\section{Stable pairs on smooth $K3$ fibrations}
 \label{sec:PTNL}
\subsection{Stable pairs on $K3$ fibrations} \label{sec:smooth}
Let $X$ be a smooth projective threefold over $\mathbb{C}$. 
By definition, a \textit{K3 fibration} is a 
morphism
\begin{align*}
\pi \colon X \to C
\end{align*}
onto a smooth projective curve $C$ whose generic 
fiber is a smooth $K3$ surface. 
If $\pi$ is a smooth morphism, then it is called a 
\textit{smooth K3 fibration}. 
We set
\begin{align*}
H_2(X, \mathbb{Z})^{\pi} \cneq 
\{ \beta \in H_2(X, \mathbb{Z}) : \pi_{\ast}\beta=0\}. 
\end{align*}
We always take curve classes contained in the above 
subgroup of $H_2(X, \mathbb{Z})$. 
An element $\beta \in H_2(X, \mathbb{Z})$ 
is called \textit{irreducible} if it is not written 
as a sum $\beta_1+\beta_2$ for non-zero 
effective curve classes $\beta_1, \beta_2$. 

For $\beta \in H_2(X, \mathbb{Z})^{\pi}$, 
we study the moduli space of stable pairs $\pP_{n}(X, \beta)$ in the sense of \cite{a17}. It parametrizes the pairs 
\begin{equation}\label{sec1:pair}
s \colon 
\O_{X} \to F
\end{equation}
 where $F$ is a pure 1-dimensional sheaf on $X$ with 
$[F]=\beta$ and $\chi(F)=n$, and the cokernel of $s$ is 0-dimensional. 
Here $[F]$ is the fundamental homology class determined by 
$F$, which is the Poincare dual of $\ch_2(F)$. 
For simplicity we set 
 $\pP=\pP_n(X,\beta)$, and denote by 
\begin{equation}\label{universal}
\Ibb =(\oO_{X \times \pP} \to \mathbb{F})
\end{equation}
the universal 
pairs, which we interpret
as an object in $D^b(X \times \pP)$. 
Note that we have the exact 
triangle on $X \times \pP$
\begin{align}\label{IOF}
\Ibb \to \oO_{X \times \pP} \to \mathbb{F} \to \Ibb[1]. 
\end{align}
Let 
\begin{align*}
\pi_{\pP}:X\times \pP\rightarrow \pP, \ \pi_{X}:X\times \pP\rightarrow X
\end{align*}
be the natural projections. 
In \cite{a17}, $\pP_n(X, \beta)$ is equipped with a perfect obstruction theory by studying the deformations of the two term 
complex (\ref{sec1:pair}) in the derived category:
\begin{thm}\label{PT}
\cite{a17}
There is a perfect obstruction theory over $\pP$ given by the following morphism in the derived category:
\begin{equation*}
 \eE^{\bullet}:=\dR\hH om_{\pi_{\pP}}(\Ibb,\Ibb \otimes 
\omega_{\pi_{\pP}})_{0}[2]\rightarrow\mathbb{L}^{\bullet}_{\pP}. 
\end{equation*} 
Here $(-)_{0}$ means taking the traceless part. 
\end{thm}
Since $\omega_X \cdot \beta=0$
for $\beta \in H_2(X, \mathbb{Z})^{\pi}$, the Riemann-Roch 
calculation easily implies that the virtual 
dimension of $\pP$ is zero. 
The stable pair invariants are then defined by 
\begin{align*}
P_{n, \beta} \cneq  
\int_{[\pP_n(X, \beta)]^{\rm{\rm{vir}}}} 1. 
\end{align*}
We define the generating series $\PT(X)_\beta$, $\PT(X)$ to be
\begin{align}\label{gen:PT}
\PT(X)_\beta \cneq 
\sum_{n \in \mathbb{Z}} P_{n, \beta}q^n, \ 
\PT(X) \cneq \sum_{\beta \in H_2(X, \mathbb{Z})^{\pi}}\PT(X)_{\beta}t^{\beta}. 
\end{align}

\subsection{Stable pairs on $K3$ surfaces}
Let $S$ be a smooth projective $K3$ surface over $\mathbb{C}$. 
For $\gamma \in H_2(S, \mathbb{Z})$, 
we can similarly define the moduli space of stable pairs 
$\pP_n(S, \gamma)$ on $S$. 
We review several 
results on the moduli space $\pP_n(S, \gamma)$.  
In \cite{a44}, Pandharipande and Thomas identify 
$\pP_{n}(S,\gamma)$ by a relative  Hilbert scheme of points:
\begin{prop}\label{pushforward}
\cite[Proposition B.8]{a44} The moduli space of stable pairs
$\pP_n(S, \gamma)$ on $S$  
is isomorphic to the relative Hilbert scheme $\Hilb^{n+\gamma^2/2}(\cC/\M)$ where $\M$ is the moduli space of pure one dimensional subschemes of $S$ in class $\gamma$ and $\cC$ denotes the universal curve over $S\times \M$.\qed
\end{prop}
We next consider the case 
that $\gamma$ is irreducible. 
In this case, the following result is 
proved in~\cite{MPT}:
\begin{prop}\label{prop:smooth}
\cite[Proposition~5]{MPT}
If $\gamma$ is irreducible, 
the moduli space $\pP_n(S, \gamma)$ is non-singular of dimension 
$n+ \gamma^2 +1$. 
It depends only upon 
$(n, \gamma^2)$ up to deformation equivalence. 
\end{prop} 
By the above proposition, 
we may write 
\begin{align}\label{P:k3}
\pP_n(S,h) \cneq \pP_n(S, \gamma)
\end{align}
for an irreducible curve class $\gamma \in H_2(S, \mathbb{Z})$
with $\gamma^2=2h-2$.  We sometimes write $\pP_n(\kt,h)$ if we do not want to specify a particular $K3$ surface with this property.
The generating series of 
$\chi(\pP_n(\kt,h))$ is computed by
Kawai-Yoshioka's formula (\ref{equ:KW}).

\subsection{Stable pairs with irreducible curve classes}
In this subsection, we assume that 
$\beta \in H_2(X, \mathbb{Z})^{\pi}$ is an irreducible 
curve class. 
Let $\pP=\pP_n(X, \beta)$ be the moduli space of stable pairs, 
and consider the perfect obstruction theory 
$\eE^{\bullet} \to \mathbb{L}^{\bullet}_{\pP}$ in 
Theorem~\ref{PT}. 
We have the following proposition: 
\begin{prop}\label{loc-free} 
In the above situation, we 
have the canonical isomorphism
\begin{align*}
\hH^{1}(\eE^{\bullet\vee})
\stackrel{\cong}{\to}
\hH om_{\pi_{\pP}}(\Ibb, \mathbb{F}\otimes\omega_{\pi_{\pP}})^{\vee}.
\end{align*}
\end{prop}
\begin{proof}
We write $\omega=\omega_{\pi_{\pP}}$ for simplicity. 
Applying 
$\dR \hH om_{\pi_{\pP}}(\Ibb,- \otimes \omega)$
 to the exact triangle 
(\ref{IOF}), 
we obtain the following exact triangle on $\pP$:
\begin{align*}
\dR \hH om_{\pi_{\pP}}
(\Ibb, \mathbb{F}\otimes \omega)\rightarrow 
\dR \hH om_{\pi_{\pP}}(\Ibb, \Ibb\otimes \omega)[1]
 \rightarrow 
\dR \hH om_{\pi_{\pP}}(\Ibb, \omega)[1].
\end{align*}
By the above triangle
and the natural morphism
\begin{align*}
\dR \pi_{\pP \ast} \omega \to 
\dR \hH om_{\pi_{\pP}}(\Ibb,\Ibb\otimes \omega)
\end{align*} 
we can form the following commutative diagram of vertical and horizontal exact triangles:
\begin{align}
\label{cone-off} 
\xymatrix@C=.3em{
   &\dR\pi_{\pP*}\omega[1]\ar[r]^{=}\ar[d]&\dR\pi_{\pP*}\omega[1]\ar[d]\\
    \dR\hH om_{\pi_{\pP}}(\Ibb,\mathbb{F}\otimes \omega)\ar[r]\ar[d]^{id} & 
\dR\hH om_{\pi_{\pP}}(\Ibb,\Ibb\otimes \omega)[1]\ar[r]\ar[d]& 
\dR\hH om_{\pi_{\pP}}(\Ibb,\omega)[1]\ar[d]\\
\dR\hH om_{\pi_{\pP}}(\Ibb,\mathbb{F}\otimes \omega)\ar[r]&
\dR\hH om_{\pi_{\pP}}(\Ibb,\Ibb \otimes\omega)_{0}[1]\ar[r]^{}  & 
\dR\hH om_{\pi_{\pP}}(\mathbb{F},\omega)[2]}
\end{align}
After dualizing, the bottom row of the diagram gives the exact triangle
 \begin{align*} 
\dR\hH om_{\pi_{\pP}}(\O,\mathbb{F})[1]\rightarrow 
\dR\hH om_{\pi_{\pP}}(\Ibb,\Ibb)_{0}[2]\rightarrow
\dR\hH om_{\pi_{\pP}}(\mathbb{F},\Ibb)[3].
\end{align*}
The 0-th cohomology of the sequence above gives the exact sequence
\begin{equation*}
\Extt^{2}_{\pi_{\pP}}(\Ibb,\Ibb)_{0}\xrightarrow{g} \Extt^{3}_{\pi_{\pP}}(\mathbb{F},\Ibb)\rightarrow \Extt^{2}_{\pi_{\pP}}(\O,\mathbb{F})
\end{equation*}
Note that the first term in the sequence above is 
$\hH^{1}(\eE^{\bullet\vee})$. Since 
the fibers of 
$\pi_{\pP}|_{\Supp(\mathbb{F})}$
are one dimensional, 
we obtain 
$\Extt^{2}_{\pi_{\pP}}(\O,\mathbb{F})= 0$ and hence the map $g$ is surjective. Now we show that the map $g$ is also injective. It is enough to prove this over any closed point $p\in \pP$. Here we will use part of the strategy of the proof of \cite[Proposition 4.4]{a17}. Over the point $p$ and after dualizing and taking cohomology, the diagram (\ref{cone-off}) gives the following diagram of the vertical and horizontal exact sequences 
\begin{equation*}
\xymatrix{&&0\ar[d]&&\\
\Ext^{2}(F, \Ib)\ar[r]\ar[d]^{=}&H^{1}(F)\ar[r]\ar[d]&\Ext^{2}(\Ib,\Ib)_{0}\ar[d]\ar@{-->}[dr]^{g_{p}}&&\\
\Ext^{2}(F,\Ib)\ar[r]&\Ext^{2}(\O_{X},\Ib)\ar[r]\ar[d]&\Ext^{2}(\Ib,\Ib)\ar[r]\ar[d]^{tr}&\Ext^{3}(F,\Ib)&\\
&H^{2}(\O_{X})\ar[r]^{=}\ar[d]&H^{2}(\O_{X})\ar[d]&&\\
&0&0&&}
\end{equation*}
where the map $g_{p}$ is induced by the map $g$ described above. 
As in \cite[Proposition 4.4]{a17}, in order to show that $\Ext^{2}(\Ib,\Ib)_{0}\xrightarrow{g_{p}}\Ext^{3}(F,\Ib)$ is injective, it is enough to show that the 
composition
\begin{align}\label{ext:compose}
\Ext^{1}(F,F) \to \Ext^2(F, \Ib) 
\to H^{1}(F)
\end{align}
is surjective. 
Here 
the first map is induced by the 
exact triangle $I^{\bullet} \to \oO_X \to F$
for the stable pair $I^{\bullet}=(\oO_X \to F)$. 


Suppose that $F$ is supported on a curve $D \subset X$
 which is irreducible by our assumption. 
We claim the inclusion and the isomorphism
\begin{align}\label{inc:iso}
\Ext^1(F,F) \supseteq  H^1(\hH om(F,F)) \cong H^1(\O_D).
\end{align}
The first inclusion follows from the local to global 
spectral sequence. 
As for the second 
isomorphism, 
let us consider the trace map 
$\hH om(F,F)\xrightarrow{tr}\O_{D}$. 
The above map is obviously surjective, and 
the stability of pairs 
implies that 
its kernel is at most zero dimensional. 
Therefore we have 
the second isomorphism in (\ref{inc:iso}). 
On the other hand, using the fact that the cokernel of the 
map $\O_D\xrightarrow{s} F$
defining the stable pair is zero dimensional, 
we can see that $H^1(\O_D) \xrightarrow{s} H^1(F)$ is surjective. 
Hence combined with (\ref{inc:iso}), 
we conclude that the map (\ref{ext:compose}) 
is also surjective. 

Now we have proved that the map $g$ is an isomorphism: 
\begin{align*}
g \colon 
\hH^{1}(\eE^{\bullet\vee}) \stackrel{\cong}{\to} 
\Extt^{3}_{\pi_{\pP}}(\mathbb{F},\Ibb). 
\end{align*}
The desired result follows from
the above isomorphism and the 
relative duality
$\Extt^{3}_{\pi_{\pP}}(\mathbb{F},\Ibb)
\cong \hH om_{\pi_{\pP}}(\Ibb,\mathbb{F}\otimes \omega_{\pi_{\pP}})^{\vee}$. 
\end{proof}

We next distinguish two kinds of components of $\pP$. 
We use the following lemma: 
\begin{lem}\label{lem:point}
Suppose that $\beta$ is irreducible. 
Then for any $(\oO_X \to F) \in \pP_n(X, \beta)$, there 
exists a unique point $p\in C$ such that 
$F$ is scheme theoretically supported on $\pi^{-1}(p)$. 
\end{lem}
\begin{proof}
Since $\beta$ is irreducible, the 
sheaf $F$ is set theoretically supported on $\pi^{-1}(p)$ for some 
$p\in C$. Let $m_p \subset \oO_{C, p}$ be the maximal ideal, and 
take $s \in m_p$. Then the sheaf homomorphism
$\cdot s \colon F \to F$ must 
be a zero map as $\beta$ is irreducible and $F$ is pure. 
Hence $F$ is an $\oO_{\pi^{-1}(p)}$-module. 
\end{proof}
The 
map 
sending 
$(\oO_X \to F)$
to $p\in C$ in Lemma~\ref{lem:point} 
defines the morphism
\begin{align*}
\rho \colon \pP \to C. 
\end{align*}
\begin{defi} \label{defn:comps}
 We call a connected component $\pP_c$ of $\pP=\pP_{n}(X,\beta)$ a 
type I component
if $\rho(\pP_c)=C$. Otherwise we call $\pP_c$ a type II component. 
\end{defi}
Below we denote by 
$\pP_{\rm{I}} \subset \pP$ the union of 
type I connected components, and 
$\pP_{{\rm{II}}} \subset \pP$
the union of type II connected components. 
We set
\begin{align*}
i_{\sS}: \sS \cneq 
X \times_{C} \pP\hookrightarrow X \times \pP. 
\end{align*}
We have the Cartesian square:
\begin{align}\label{Car}
\xymatrix{
\sS \ar[r]^{i_{\sS}} \ar[d]_{\rho_{\sS}} \ar@{}[dr]|\square
& X \times \pP \ar[d]^{(\pi, \rho)}\\
C \ar[r] & C \times C. 
}
\end{align}
Here the bottom arrow is the diagonal. 
The universal one dimensional sheaves $\mathbb{F}$ in (\ref{universal}) 
is written as $i_{\sS \ast}\mathbb{G}$ for 
a sheaf $\mathbb{G}$ on $\sS$. 
We denote by 
$\Ibb_{\sS}$ the universal pairs on $\sS$:
\begin{align*}
\Ibb_{\sS} \cneq (\oO_{\sS} \to \mathbb{G}). 
\end{align*}
Also we denote by $\pi'_{\pP}$ and $\pi'_X$ the compositions of $i_{\sS}$ with $\pi_{\pP}$ and $\pi_X$ respectively, i.e. they are projections:
\begin{align*}
\pi_{\pP}' \colon \sS \to \pP, \ \pi_{X}' \colon \sS \to X. 
\end{align*}
By \cite[Proposition 3.4]{a112} and the standard techniques of \cite{a10},
the 
type I (resp.~II)
components can be equipped with a 
relative (absolute) obstruction theory: 
\begin{thm}\label{thm:trunc2}
(i) 
There exists a $\rho$-relative perfect 
obstruction theory over $\pP_{\rm{I}}$ given 
by 
\begin{align*}
\gG^{\bullet}_{\rm{I}}
\cneq \dR \hH om_{\pi_{\pP}'}(\mathbb{G},\Ibb_{\sS}
\otimes\omega_{\pi'_{\pP}} )|_{\pP_{\rm{I}}}[2]
\rightarrow\mathbb{L}^{\bullet}_{\pP_{\rm{I}}/C}.
\end{align*} 
(ii) There exists a perfect obstruction theory over $\pP_{\rm{II}}$
given by 
\begin{align*}
\gG^{\bullet}_{\rm{II}}
\cneq \dR \hH om_{\pi_{\pP}'}(\mathbb{G},\Ibb_{\sS}
\otimes\omega_{\pi'_{\pP}} )|_{\pP_{\rm{II}}}[2]
\rightarrow\mathbb{L}^{\bullet}_{\pP_{\rm{II}}}.
\end{align*}
\end{thm}

\subsection{Type I component}
Now we assume that $\pi \colon X \to C$ is a smooth $K3$ fibration and 
$\beta$ an irreducible curve class. 
Here we investigate the contributions
of the type I components to the virtual classes. 
We have the following statement: 
\begin{prop}\label{smoothness} \cite[Proposition C.2 and Lemma C.7]{a44}
The type I components $\pP_{\rm{I}}$ 
is non-singular of dimension $n+2g$, where 
$g$ is the arithmetic genus of the support of a stable pair in $\pP_{\rm{I}}$. 
Moreover, we have the isomorphism of the tangent bundles:
\begin{align*}
\tT_{\pP_{\rm{I}}} \cong \hH om_{\pi_{\pP}}(\Ibb,\mathbb{F})|_{\pP_{\rm{I}}}, \quad
\tT_{\pP_{\rm{I}}/C}\cong \hH om_{\pi'_{\pP}}(\Ibb_{\sS},\GG)|_{\pP_{\rm{I}}}. 
\end{align*}
\end{prop}
Let
$\kK$ be the line bundle on $C$ given by
\begin{align*}
\kK \cneq \pi_{\ast} \omega_{X/C}. 
\end{align*}
\begin{prop}\label{prop:SES}
We have the exact sequence of vector bundles on 
$\pP_{\rm{I}}$:
\begin{align*}
0 \to \rho^{\ast} \kK^{\vee} \to 
&\hH om_{\pi_{{\pP}}}(\Ibb,\mathbb{F}\otimes\omega_{\pi_{{\pP}}})^{\vee}
|_{\pP_{\rm{I}}} \\
& \to\hH om_{\pi'_{{\pP}}}(\Ibb_{\sS},
\GG\otimes\omega_{\pi'_{{\pP}}})^{\vee}|_{\pP_{\rm{I}}}
\otimes \rho^{\ast}T_C \to 0.
\end{align*}
\end{prop}
\begin{proof}
Since $\mathbb{F}=i_{\sS \ast} \mathbb{G}$, 
we have the
exact triangle
\begin{align}\label{tri:GI}
\GG(-\sS)\rightarrow \dL i_{\sS}^{*}\Ibb\rightarrow \Ibb_{\sS}.
\end{align}
We apply 
$\dR \hH om_{\pi'_{\pP}}(-, \GG \otimes i_{\sS}^{\ast}
\omega_{\pi_{\pP}})$
 to the above exact triangle. 
By taking the cohomologies
and restricting to $\pP_{\rm{I}}$, we get the exact sequence
\begin{align*}
0
\rightarrow \hH om_{\pi'_{\pP}}(\Ibb_{\sS},\GG \otimes
i_{\sS}^{\ast}\omega_{\pi_{\pP}})|_{\pP_{\rm{I}}}
&\rightarrow \hH om_{\pi_{\pP}}(\Ibb, \mathbb{F} \otimes 
\omega_{\pi_{\pP}})|_{\pP_{\rm{I}}} \\
&\rightarrow \hH om_{\pi'_{\pP}}(\GG(-\sS),\GG\otimes i_{\sS}^{\ast}\omega_{\pi_{\pP}})|_{\pP_{\rm{I}}}. \end{align*}
By Proposition~\ref{smoothness}, 
all 
the sheaves in the 
above sequences are vector bundles. 
We investigate the fibers of the 
vector bundles in the sequence above. 
By the fiberwise stability of $\GG$, 
the bundle
$\hH om_{\pi'_{\pP}}(\GG(-\sS),\GG\otimes i_{\sS}^*\omega_{\pi_{\pP}})
|_{\pP_{\rm{I}}}$
is a line bundle on $\pP_{\rm{I}}$. 
Therefore by 
Proposition~\ref{smoothness} and 
the definition of type I component, 
the above sequence 
must be a short exact sequence. 
After dualizing, we obtain the short exact sequence 
\begin{align*}
0\to \Extt^{2}_{\pi'_{\pP}}(\GG,\GG)|_{\pP_{\rm{I}}}
&\to 
\hH om_{\pi_{\pP}}(\Ibb,\mathbb{F}\otimes\omega_{\pi_{\pP}})^{\vee}|_{\pP_{\rm{I}}} \\
&\to\hH om_{\pi'_{\pP}}(\Ibb_{\sS},\GG\otimes
i_{\sS}^{\ast}\omega_{\pi_{\pP}})^{\vee}|_{\pP_{\rm{I}}}\to 0.
\end{align*}
Here the first term is obtained by applying 
the Grothendieck duality and the adjunction formula $\omega_{\pi'_{\pP}}\cong i_{\sS}^{*}\omega_{\pi_{\pP}}\otimes \mathcal{O}_{\sS}(\sS)$.
Now we have $$\Extt^{2}_{\pi'_{\pP}}(\GG,\GG)|_{\pP_{\rm{I}}} \cong
R^2 \hH om_{\pi'_{\pP}}(\GG,\GG)|_{\pP_{\rm{I}}}$$ which is identified with $R^{2}{\pi'_{\pP \ast}}\mathcal{O}_{\sS}|_{\pP_{\rm{I}}}$
 via the trace map using the fiberwise stability of $\GG$ again. By the Grothendieck duality and the adjunction formula again and noting that $\omega_{\pi'_{\pP}}=\pi_X^{'\ast} \omega_{X/C}$, 
we obtain 
$$R^{2}{\pi'_{\pP \ast}}\mathcal{O}_{\sS}|_{\pP_{\rm{I}}}
\cong (R^{0}{\pi'_{\pP \ast}}
\omega_{\pi'_{\pP}})^{\vee}|_{\pP_{\rm{I}}}
\cong \rho^* \mathcal{K}^\vee.$$
Finally, again using the 
adjunction 
and noting 
$\oO_{\sS}(\sS) \cong \rho^{\ast}T_C$
by the diagram (\ref{Car}), we have the isomorphism
\begin{align*}
\hH om_{\pi'_{\pP}}(\Ibb_{\sS},\GG \otimes
i_{\sS}^{\ast}\omega_{\pi_{\pP}})^{\vee}|_{\pP_{\rm{I}}}
\cong 
\hH om_{\pi'_{\pP}}(\Ibb_{\sS}, \GG\otimes\omega_{\pi'_{\pP}})^{\vee}
|_{\pP_{\rm{I}}}
\otimes \rho^{\ast}T_C. 
\end{align*}
Therefore we obtain the desired exact sequence. 
\end{proof}
Using the above proposition, we have the following statement: 
\begin{prop}\label{prop:typeI}
The restriction
 of the virtual 
class of the obstruction theory $\eE^\bullet$ in Theorem \ref{PT} to 
the 
type I
components $\pP_{\rm{I}}$ 
is given by  
\begin{align*}
[\pP_{\rm{I}}, \eE^{\bullet}]^{\rm{vir}}
=[\pP_{\rm{I}}]\cap \big(c_{1}(\rho^{*}\mathcal{K}^{\vee})\cup c_{top}(\Omega_{\pP_{\rm{I}}/C})\big). 
\end{align*}
\end{prop}
\begin{proof}
By Propositions \ref{loc-free} and \ref{smoothness}, $\pP_{\rm{I}}$ is smooth with the obstruction sheaf $\hH om_{\pi_{\pP}}(\mathbb{I}^{\bullet}, \mathbb{F} \otimes 
\omega_{\pi_{\pP}})^{\vee}|_{\pP_{\rm{I}}}$. Using the short exact sequence 
in Proposition~\ref{prop:SES}, we can write 
\begin{align} \label{equ:ctop}
&c_{top}(\hH om_{\pi_{\pP}}(\mathbb{I}^{\bullet}, \mathbb{F} \otimes 
\omega_{\pi_{\pP}})^{\vee}|_{\pP_{\rm{I}}}) \\
\notag
&=c_1(\rho^{\ast} \kK^{\vee}) \cup 
c_{top}(
\hH om_{\pi'_{\pP}}(\Ibb_{\sS},\GG\otimes\omega_{\pi'_{\pP}})^{\vee}
|_{\pP_{\rm{I}}}
\otimes \rho^{\ast}T_C). 
\end{align}
By the following lemma and using the fact that the intersection product of any 
two classes on $A^1_\Q(C)$ is zero, we can see that only 
$$c_1(\rho^{\ast} \kK^{\vee}) \cup c_{top}
(\hH om_{\pi'_{\pP}}(\Ibb_{\sS},\GG)^{\vee}|_{\pP_{\rm{I}}})=
c_1(\rho^{\ast} \kK^{\vee}) \cup 
c_{top}(\Omega_{\pP_{\rm{I}}/C})$$ contributes to the formula \eqref{equ:ctop}.
Here the above identity is due to Proposition~\ref{smoothness}. 
 Therefore, 
the desired equality follows from \cite[Proposition 5.6]{BF}.  
\end{proof}

\begin{lem}\label{ctop}
Suppose that $\pP_0 \subset \pP_{\rm{I}}$ is a type I component of dimension $l$. Then, we have the following relation in $A^{*}_\Q(\pP_0)$:
\begin{equation*}
c_{top}(\mathcal{H}om_{\pi_{\pP_{0}}}(\Ibb_\sS,\GG\otimes\omega_{\pi'_{\pP_0}})^{\vee})= c_{l}(\mathcal{H}om_{\pi'_{\pP_0}}(\Ibb_\sS,\GG)^{\vee})+A \cdot \rho^*(B)
\end{equation*}
for some $A \in A_\Q^{l-1}(\pP_0)$ and $B \in A_\Q^1(C)$.
\end{lem}
\begin{proof}
Since $\omega_{X/C}$ is fiberwise trivial, it must be the pullback of a line bundle $M$ on $C$. So we can write $\omega_{\pi'_{\pP_0}}=\pi'^*_X\omega_{X/C}=\pi^*_{\pP_0}\circ \rho^*M$.
The Grothendieck-Riemann-Roch formula gives: 
\begin{align*}
&\operatorname{ch}\bigg(\sum_{j=0}^1(-1)^j\Extt^j_{\pi'_{\pP_0}}(\Ibb_{\sS},\GG\otimes\omega_{\pi'_{\pP_0}})^{\vee}\bigg)\\&=\pi'_{\pP*}\bigg(\operatorname{ch}(\Ibb_\sS) \cdot \operatorname{ch}(\mathbb{G})^\vee \cdot\operatorname{ch}(\omega_{\pi'_{\pP_0}})^\vee\cdot \operatorname{td}(X {_{\pi}}\times_\rho \pP_0) \bigg)\\&=(1-\rho^*c_1(M)) \cdot \operatorname{ch}\bigg(\sum_{j=0}^1(-1)^j\Extt^j_{\pi'_{\pP_0}}(\Ibb_\sS,\mathbb{G}))^\vee\bigg).
\end{align*} 
By (C.6) in the proof of \cite[Proposition C.2]{a44}, the fibers of 
the sheaves $$\Extt^1_{\pi'_{\pP_0}}(\Ibb_{\sS},\GG\otimes\omega_{\pi'_{\pP_0}})^{\vee} \quad \text{and} \quad \Extt^1_{\pi'_{\pP_0}}(\Ibb_{\sS},\GG)^{\vee}$$ at any closed point $I^\bullet_S=(\O_S\to G) \in \pP_0$ is naturally identified with $$\Ext^1(I^\bullet_S,G)\cong H^2(\O_S)$$ for a $K3$ fiber $S\subset X$. Therefore, both sheaves are line bundles pulled back from the base curve $C$,  thus their
 Chern characters are the pull backs of the classes from $A_\Q^*(C)$. The lemma is then proven by an inductive argument on $n$.
\end{proof}

\subsection{Type II component}
Let 
$\pP_c \subset \pP_{\rm{II}}$ be
 a type II component. 
Then there exists a point $p\in C$ such that $\rho(\pP_c)=p$. 
Let us set $S=\pi^{-1}(p)$ and $i \colon S \hookrightarrow X$
is the inclusion. 
Then there exists unique $\gamma \in H_2(S, \mathbb{Z})$ 
with $i_{\ast}\gamma=\beta$ and a 
closed embedding
\begin{equation}\label{equn}
\pP_n(S, \gamma) \subset \pP_c. 
\end{equation}
\begin{defi}
A type II component $\pP_c$ is called isolated if the 
embedding (\ref{equn}) is an isomorphism. 
\end{defi}

\begin{rmk}\label{rmk:tan}
Since $\pP_n(S, \gamma)$ is non-singular by Proposition~\ref{prop:smooth}, 
the embedding (\ref{equn}) is an isomorphism if and only if 
(\ref{equn}) induces isomorphisms of
tangent spaces at every points in $\pP_n(S, \gamma)$.  
\end{rmk}

%
%
%

We have the following lemma: 
\begin{lem}\label{lem:iso}
Suppose that any component of $\pP_{\rm{II}}$ is isolated. 
Then we have the isomorphism of vector bundles:
\begin{align*}
\Omega_{\rm{\pP_{\rm{II}}}}
\stackrel{\cong}{\to}
\hH om_{\pi_{\pP}}(\mathbb{I}, \mathbb{F} \otimes \omega_{\pi_{\pP}})^{\vee}
|_{\pP_{\rm{II}}}. 
\end{align*}
\end{lem}
\begin{proof}
By the definition of type II components and 
Remark~\ref{rmk:tan}, the following 
morphism is an isomorphism
\begin{align*}
T_{\pP_{\rm{II}}} \cong \hH om_{\pi_{\pP}'}(\mathbb{I}_{\sS},
 \mathbb{G})|_{\pP_{\rm{II}}}
\stackrel{\cong}{\to}
\hH om_{\pi_{\pP}}(\mathbb{I}, \mathbb{F})|_{\pP_{\rm{II}}}. 
\end{align*}
Here the first isomorphism is due to Theorem~\ref{thm:trunc2} (ii). Since $\omega_X$ restricted to any $K3$ fiber is trivial, and for any component $\pP_c\subset \pP_{\rm{II}}$, $\mathbb{F}|_{\pP_c\times X}$ is supported on $\pP_c\times S$ for some $K3$ fiber $S\subset X$, 
we have the isomorphism $$\mathbb{F}\otimes \omega_{\pi_{\pP}}|_{\pP_{\rm{II}}\times X}\cong \mathbb{F}|_{\pP_{\rm{II}}\times X}.$$
Therefore we obtain the lemma.  
\end{proof}
By the above lemma and Proposition~\ref{loc-free}, 
the contributions
to virtual classes from the isolated 
type II components are 
given as follows:  
\begin{cor}\label{decompos}
In the situation of Lemma~\ref{lem:iso}, 
the restriction of the virtual class of the obstruction 
theory $\eE^\bullet$ in Theorem \ref{PT} to 
the type II components $\pP_{\rm{II}}$ is given by
\begin{align*}
[\pP_{\rm{II}}, \eE^{\bullet}]^{\rm{vir}}&=[\pP_{\rm{II}}]
\cap c_{top}(\Omega_{\pP_{\rm{II}}}).
\end{align*} 
\end{cor}

\subsection{Generating series of stable pair invariants}
Following \cite{a90}, let 
\begin{align}\label{loc:sys}
\V = R^2\pi_*(\Z)\to C
\end{align}
 be the rank 22 local system determined by the $K3$ fibration $\pi$. 
We denote by $\V_c$ the fiber of 
(\ref{loc:sys}) at $c \in C$.
 Let $\mathcal{H}^{\V}$ denote the $\pi$-relative moduli space of Hodge structures as in \cite[Section 1.4]{a90} \footnote{In \cite[Section 1.4]{a90} this is denoted by $\mathcal{M}^{\V}$.}, i.e. 
there is a map 
\begin{align*}
\hH^{\vV} \to C
\end{align*} such that each fiber
$\hH^{\vV_c}$ 
 at $c\in C$ is 
the moduli space of weight two Hodge structures on 
$\vV_{c} \otimes \mathbb{C}=H^2(S, \mathbb{C})$ 
for $S=\pi^{-1}(c)$. 
 There exists a section map \begin{equation}\label{equ:section}\sigma:C\rightarrow \mathcal{H}^{\V}\end{equation} which is determined by the Hodge structures of the fibers of $\pi$:
\begin{equation}\label{section-map}
\sigma(c)=[H^{2, 0}(S) \subset H^2(S, \mathbb{C})]
\in \mathcal{H}^{\V_{c}}, \ S=\pi^{-1}(c).
\end{equation}
For $c \in C$, 
an irreducible class $\beta \in H^2(X,\Z)^\pi$ and $h \in \Z$, 
we define
$$\V_c(h,\beta)=\{0\neq \gamma \in \V_c \mid \gamma^2=2h-2, 
i_{\ast}\gamma=\beta\}.$$
Also let $\operatorname{B}_c(h,\beta)\subseteq \V_c(h,\beta)$ be the subset of $\gamma \in \V_c(h,\beta)$ where $\gamma$ is a $(1,1)$ class on $S=\pi^{-1}(c)$. Then $\operatorname{B}_c(h,\beta)$ is finite by~\cite[Proposition~1]{a90}. 
The subset 
\begin{align*}
\operatorname{B}(h,\beta) \cneq \bigcup_{c\in C} \operatorname{B}_c(h,\beta)\subset \V
\end{align*}
 can be decomposed into 
$$\operatorname{B}(h,\beta)
=\operatorname{B}_{\rm{I}}(h,\beta)\coprod \operatorname{B}_{\rm{II}}(h,\beta)$$ where the first component defines a finite local system $\epsilon:\operatorname{B}_{\rm{I}}(h,\beta)\to C$, and the second component is an isolated set. Let $\pP_\epsilon$ be the connected component of 
the stable pair moduli
space $\pP=\pP_n(X, \beta)$ corresponding to the local system $\epsilon$, and 
let 
\begin{align*}
P_{n,\epsilon}=\int_{[\pP_\epsilon, \eE^\bullet]^{\rm{vir}}}1
\end{align*}
 be the contribution of this component to $P_{n,\beta}$. Note that $\pP_\epsilon$ is a type I component of $\pP_{n}(X,\beta)$ in the sense of Definition \ref{defn:comps}. By 
Proposition~\ref{prop:typeI}, 
we have 
\begin{equation}
[\pP_{\epsilon}, \eE^{\bullet}]^{\rm{vir}}=[\pP_{\epsilon}]\cap \left(c_{1}(\rho^{*}\mathcal{K}^{\vee})\cup c_{top}(\Omega_{\pP_{\epsilon/C}})\right). 
\end{equation}

On the other hand, 
suppose that 
 $\alpha \in \operatorname{B}_{\rm{II}}(h,\beta)$ 
supported on the fiber $S$ is a 
result of a \emph{transversal intersection} (i.e. with intersection multiplicity 1) of $\sigma(C)$ with a Noether-Lefschetz divisor. 
Then 
the corresponding connected component $\pP_\alpha$ of $\pP_{n}(X,\beta)$ is an isolated type II component in the sense of Definition \ref{defn:comps}, and 
it is isomorphic to $\pP_{n}(S,h)$ given in (\ref{P:k3}). 
By Corollary \ref{decompos}, we have 
\begin{equation}\label{virtual2} 
[\pP_\alpha, \eE^\bullet]^{\rm{vir}}=[\pP_\alpha]\cap c_{top}(\Omega_{\pP_\alpha}).
\end{equation} 
We define 
$$P_{n,\alpha}=\int_{[\pP_\alpha, \eE^\bullet]^{\rm{vir}}}1$$
 to be the contribution to $P_{n,\beta}$ of this component. For any integer $h \in \Z $, the Noether-Lefschetz number 
\begin{align*}
NL^\pi_{h,\beta} \in \mathbb{Z}
\end{align*}
 was defined in \cite{a90} by intersecting $\sigma(C)$ with the $\pi$-relative Noether-Lefschetz divisor in $\mathcal{H}^\V$ associated to $h$ and $\beta$. Informally, $NL^\pi_{h,\beta}$ is the number of the fibers $S$ of $\pi$ for which there exists a $(1,1)$ class $\gamma \in H^2(S,\Z)$ such that $$\gamma^2=2h-2 \;\;\text{and}\;\;  i_*\gamma=\beta.$$



The following theorem expresses the stable pair invariants of $X$ in terms of the Euler characteristics of of the moduli spaces of stable pairs on the fibers and the Noether-Lefschetz numbers:
\begin{thm} \label{thm:main formula}
For a smooth $K3$ fibration $\pi:X\to C$ and an irreducible class $\beta\in H_2(X,\Z)^\pi$, we have 
\begin{align*}
\PT(X)_\beta&=\sum_{h=0}^{\infty}\sum_{n=1-h}^{\infty}(-1)^{n-1}\chi(\pP_{n}(\kt,h))\cdot NL^{\pi}_{h,\beta}\; q^{n}. \notag\\
\end{align*}
Here, the Euler characteristics $\chi(\pP_{n}(\kt,h))$ are 
determined by  Kawai-Yoshioka's formula \eqref{equ:KW}.
\end{thm}

\begin{proof}
The proof follows the same ideas as the proof of~\cite[Theorem~1]{a90}. We compare the contributions of $B_{\rm{I}}$ and $B_{\rm{II}}$ to 
$P_{n,\beta}$ and the Noether-Lefschetz numbers. Suppose that $\epsilon$ is a local system giving rise to $B_I(h,\beta)$ as above. Then, we can write 
\begin{align*}
P_{n,\epsilon}&=\int_{\pP_\epsilon}c_{top}(\Omega_{\pP_{\epsilon/C}})\cdot c_1(\rho^*\K^\vee)\\
&=(-1)^{n-1}\chi(P_{n}(\kt,h))\cdot \int_{B_{\rm{I}}(h,\beta)} 
c_1(\epsilon^*\K^\vee).
\end{align*}  By virtue of \eqref{section-map}, it is shown 
in~\cite[Theorem~1]{a90} that 
the integration 
$$\int_{B_{\rm{I}}(h,\beta)} c_1(\epsilon^{\ast}\K^\vee)$$ gives the contribution of $B_{\rm{I}}(h,\beta)$ to $NL^\pi_{h,\beta}$.  

Using the deformation invariance of the stable pair invariants, 
we may assume that any $\alpha \in B_{\rm{II}}(h,\beta)$ is a result of a transversal intersection of $\sigma(C)$ with a Noether-Lefschetz divisor\footnote{Here as in \cite[pg. 23]{a90}, one may need to make a local holomorphic perturbation of
the section $\sigma$ to make it transversal to the Noether-Lefschetz divisor (i.e. with the local intersection multiplicity 1). The $K3$ fibration after perturbation will be over an analytic curve, but the stable pairs under consideration  will remain on the $K3$ fibers, which
are algebraic. The corresponding type II component of the moduli space of stable pairs is always algebraic and compact and lies over the interior of
a small (analytic) open set in the curve that is being perturbed. So the usual deformation invariance of the intersection numbers gives the deformation invariance of the contribution of the stable pair invariants of this type II component.}. As a result, the contribution of $\alpha$ to $NL^\pi_{h,\beta}$ is exactly 1, and moreover,  the corresponding component $\pP_\alpha$ is an isolated type II component. 
Therefore by Corollary~\ref{decompos}, we have
$$
P_{n,\alpha}=\int_{\pP_\alpha}c_{top}(\Omega_{\pP_\alpha})=(-1)^{n-1}\chi(\pP_{n}(\kt,h)).
$$
The above arguments prove the desired identity. 
\end{proof}


\subsection{Reducible curve classes} \label{sec:reducible}
In this section we study a special analog of Theorem \ref{thm:main formula} in which the class $\beta \in H_2(X, \mathbb{Z})^{\pi}$ is allowed to be reducible.  

Let $\pP_{c}$ be an isolated type II component of 
$\pP_n(X, \beta)$. This means that we have an isomorphism 
$\pP_{c}\cong \pP_n(S, \gamma)$ where 
$\beta=i_{\ast}\gamma$ and $i:S\hookrightarrow X$ is a $K3$ fiber of $X$. Here, $\gamma$ is no longer needed to be irreducible, but we always assume that\footnote{This is the case for example if $\gamma$ is big and nef (see \cite[Proposition 3.1]{HuK3}).} \begin{equation} \label{H1vanish} H^1(L)=0\text{ for any line bundle on $S$ with } c_1(L)=\gamma.\end{equation}  
As $\gamma$ is not an irreducible class, 
the moduli space
$\pP_n(S,\gamma)$ may be singular and 
have several irreducible components.
However we have the following lemma:
\begin{lem} \label{dimcount}
Suppose that $\gamma$ satisfies \eqref{H1vanish}, then the dimension of each irreducible component of $\pP_n(S,\gamma)$ does not exceed $v:=n+\gamma^2+1$.
\end{lem}
\begin{proof}
In the notation of Proposition \ref{pushforward}, the dimension of 
$\M$ is $\gamma^2/2+1$, and for a fixed curve $D$ in class $\gamma$ the dimension of the components of the Hilbert scheme are at most $n+\gamma^2/2$. These two claims prove the lemma. The first claim follows because in this case $\M$ is the linear system of curves in class $\gamma$ on $S$ satisfying \eqref{H1vanish}. The second claim can be verified by analyzing the fibers of the Hilbert to Chow morphism from the Hilbert scheme of $k$ points on $D$ $$HC:\hilb^k(D)\to \text{Sym}^k(D).$$ We know that the dimension of the punctual Hilbert scheme of $a$ points supported on a fixed point of $S$ is equal to $a-1$ (see \cite{joel}). Since $D \subset S$ the dimension of the punctual Hilbert scheme of $a$ points supported on a fixed point of $D$ does not exceed $a-1$. Now given a $d$-dimensional diagonal $\Delta_d \subset \text{Sym}^k(D)$ corresponding to the partition $k=a_1+\cdots+ a_d$ and a point $p \in \Delta_d$, by what we said above, the dimension of $HC^{-1}(p)$ does not exceed $\sum_i (a_i-1)=k-d$. From this the claim follows.  
\end{proof}

By \cite[Proposition 3.4]{a112}
(also see Theorem~\ref{thm:trunc2}~(ii)), 
 the component 
$\pP_c=\pP_n(S,\gamma)$ is equipped with a perfect obstruction theory 
\begin{align*}
\gG^{\bullet}
\cneq \dR \hH om_{\pi_{\pP}'}(\mathbb{G},\Ibb_{\sS}
\otimes\omega_{\pi'_{\pP}} )|_{\pP_{c}}[2]
\rightarrow\mathbb{L}^{\bullet}_{\pP_{c}}.
\end{align*}
Note that the 
obstruction sheaf 
 $\hH^1(\gG^{\bullet \vee})$
 of $\gG^{\bullet}$
admits a surjection to 
the trivial vector bundle  
$\O_{\pP_{c}}$, given 
by the natural surjection
\begin{align}\label{surj}
\Hom(I_{S}^{\bullet}, G[1]) \twoheadrightarrow \Ext_S^2(G, G)
\twoheadrightarrow \mathbb{C}
\end{align}
at the fiber of 
$I_{S}^{\bullet}=(\oO_{S }\to G) \in \pP_c$. 
The second 
morphism of (\ref{surj})
is dual to 
$\mathbb{C} \cdot \id \subset \Hom(G, G)$. 
  By removing the trivial factor $\oO_{\pP_c}$ from 
the obstruction theory $\gG^{\bullet}$, 
Kool-Thomas~\cite{a112} constructs a $v$-dimensional reduced 
virtual cycle
\begin{align*}
[\pP_n(S,\gamma)]^{\rm{red}}
\in A_{v}(\pP_n(S, \gamma)).
\end{align*}
Here $v$ is given in Lemma~\ref{dimcount}.
Let $\eE^{\bullet}$ be the obstruction theory 
in Theorem~\ref{PT}
restricted to $\pP_c$.  
We define the following element in the $K$-group
\begin{align*}
\vV =-\eE^{\bullet \vee}+\gG^{\bullet \vee}
+\oO_{\pP_c} \in K(\pP_c). 
\end{align*}
Then $\vV$ is of constant rank $v$. 
\begin{prop}\label{prop:red} Suppose that $\gamma$ satisfies \eqref{H1vanish} and $v$ is as in Lemma~\ref{dimcount},
then we have the following identity: 
\begin{align}\label{id:red}
[\pP_{c}, \eE^\bullet]^{\rm{vir}}
=c_v(\vV) \cap [\pP_n(S, \gamma)]^{\rm{red}}. 
\end{align}
\end{prop}
\begin{proof}
The obstruction theories $\eE^\bullet$ and 
$\gG^\bullet$ are related by the following two natural exact triangles:
 \begin{align}\label{cone-off3} 
\dR\hH om_{\pi_{\pP}}(\Ibb, \mathbb{F})\to 
\dR\hH om_{\pi_{\pP}}(\Ibb, \Ibb)_{0}[1]
\rightarrow  \dR\hH om_{\pi_{\pP}}(\mathbb{F},\O_{X\times \pP})[2],
\end{align}
and
\begin{align}\label{cone-off4} 
\dR\hH om_{\pi_{\pP}'}(\Ibb_\sS,\GG)\to 
\dR\hH om_{\pi_{\pP}}(\Ibb,\mathbb{F})
\rightarrow \dR\hH om_{\pi'_{\pP}}(\GG(-\sS),\GG).
\end{align}
Here the triangle (\ref{cone-off3})
is obtained similarly to 
the bottom triangle of the diagram (\ref{cone-off}), 
replacing $\omega$ by $\oO_{X \times \pP}$.
The triangle (\ref{cone-off4})
follows from applying $\dR \hH om_{\pi_{\pP}'}(-, \mathbb{G})$
to the triangle (\ref{tri:GI}). 
The complexes $\eE^{\bullet \vee}$ and $\gG^{\bullet \vee}$ can be identified in the derived category with 2-term complexes 
\begin{align*}
\eE_0\to \eE_1 \quad \text{and} \quad \gG_0\to \gG_1
\end{align*}
 of vector bundles. Their $K$-group classes are then respectively 
$\eE_0-\eE_1$ and $\gG_0-\gG_1$. We define the element $\vV'$ 
of the $K$-group by 
\begin{align*}
\vV'=\eE_1-\eE_0-(\gG_1-\gG_0).
\end{align*}
Then we have $\vV=\vV'+\oO_{\pP_c}$. 
 By the exact triangles (\ref{cone-off3}), (\ref{cone-off4}) above, the fiber of $\vV'$ over a closed point $(\O_S\to G)\in \pP_c$
 is naturally given by $\chi(G)-
\chi(G, G)$. 

Using these facts
and noting that the summation of $\oO_{\pP_c}$ does not affect the 
total Chern class, the result of~\cite[Theorem 4.6]{Siebert2}
implies 
 \begin{align} \label{equ:isocomp}
 [\pP_{c}, \eE^\bullet]^{\rm{vir}}&=\{c(\eE_1-\eE_0)\cap c_F(\pP_{c})\}_0\\ \notag
&=\{(c(\vV)\cup c(\gG_1-\gG_0))\cap c_F(\pP_n(S,\gamma))\}_0.
\end{align}
Here $\{-\}_{r}$ means the $r$-dimensional part, and 
$c_F$ denotes Fulton's canonical class. 
Similarly, we have
\begin{align*}
[\pP_n(S, \gamma)]^{\rm{red}}=
\{c(\gG_1-\gG_0) \cap c_F(\pP_n(S, \gamma))\}_{v}.
\end{align*}
By 
the discussion in~\cite[Section 4.1]{Siebert2}, 
we know that 
\begin{align}\label{c:vanish}
\{c(\gG_1-\gG_0)\cap c_F(\pP_n(S,\gamma))\}_r=0
\end{align}
 for $r<v-1$.
It is also 0 for $r=v-1$ because as mentioned earlier the obstruction theory $\gG^\bullet$ contains a trivial factor, hence 
$[\pP_c, \gG^{\bullet}]^{\rm{vir}}=0$. 
The vanishing (\ref{c:vanish}) also holds for $r>v$ because $\dim \pP_n(S,\gamma) \le v$ by Lemma~\ref{dimcount}. 
Therefore by the dimension reason, we obtain the
desired identity (\ref{id:red}). 
\end{proof}

\begin{cor} \label{irred}
In the situation of Proposition \ref{prop:red}, if $\gamma$ is irreducible then \begin{align*}[\pP_n(S,\gamma)]&=[\pP_n(S,\gamma)]^{\rm{red}},\\ c_v(\vV)\cap[\pP_n(S,\gamma)]^{\rm{red}}&= c_v(\Omega_{\pP_n(S,\gamma)})\cap [\pP_n(S,\gamma)].\end{align*}
\end{cor}
\begin{proof} In this case by Proposition \ref{prop:smooth}, $\pP_n(S,\gamma)$ is smooth of dimension $v$, which is the same as the virtual dimension of $[\pP_n(S,\gamma)]^{\rm{red}}$. This proves the first equality. The second equality follows by noting that both sides give $[\pP_{c}, \eE^\bullet]^{\rm{vir}}$ by Proposition \ref{prop:red} and Corollary \ref{decompos}.\end{proof}

\begin{cor} \label{deform}
In the situation of Proposition \ref{prop:red}, suppose that there is a smooth deformation of $S$ to a $K3$ surface in which $\gamma$ becomes irreducible. Then $$\int_{[\pP_n(S,\gamma)]^{\rm{red}}}c_v(\vV)=(-1)^v\chi(\pP_n(\kt,h)),$$ where $h=\gamma^2/2+1$.
\end{cor}
\begin{proof} Since we are working with an isolated type II component, the virtual bundle $\vV$ depends on $S$ only. To see this first note that by Verdier duality applied to the closed immersion $S\times \pP_c\to X\times \pP_c$ and that $\omega_S=\O_S\cong \omega_X|_S$, we have  $$\dR\hH om_{\pi_{\pP}}(\mathbb{F},\O_{X\times \pP})\cong  \dR\hH om_{\pi'_{\pP}}(\mathbb{G},\O_{S\times \pP})[-1].$$ Using this and the exact triangles \eqref{cone-off3} and \eqref{cone-off4}, it is clear that the K-group class $[\dR\hH om_{\pi_{\pP}}(\Ibb, \Ibb)_{0}]$ only depends on the data on $S$, and hence the same is true for $\vV$. The corollary now follows from Corollary \ref{irred} and deformation invariance of the reduced virtual cycle for $(1,1)$ classes. 
\end{proof}

In the following special situation we can prove an analog of Theorem \ref{thm:main formula} for possibly reducible classes $\beta$: 

\begin{thm} \label{rmk:lasteq}
For a smooth $K3$ fibration $\pi:X\to C$ suppose that the class $\beta\in H_2(X,\Z)^\pi$ is such that
$\pP_n(X, \beta)$ consists of only isolated type II components. For any such component $\pP_n(S, \gamma)$ as above,  suppose that $\gamma$ satisfies the condition \eqref{H1vanish}.
Then, 
\begin{align}\label{form:red}
P_{n, \beta}=\sum_{h=0}^{\infty}
(-1)^{n-1} \chi(\pP_n(\kt,h)) \cdot 
NL^{\pi}_{h,\beta}.
\end{align} 
\end{thm}
\begin{proof}
Since all the components of $\pP_n(X, \beta)$ are isolated type II, for any such component $\pP_n(S, \gamma)$, we necessarily have  $\gamma=\gamma_1+\cdots +\gamma_r$ where $\gamma_i$'s are distinct irreducible classes\footnote{If any of $\gamma_i$'s is not reduced then one could have a stable pair whose support infinitesimally thickens outside of the supporting $K3$ fiber $S$ contradicting the assumption that all the components  of $\pP_n(X, \beta)$ are isolated type II.}. In particular, $\gamma$ is a primitive class and hence one can find a deformation of $S$ as in Corollary \ref{deform}. This together with similar argument as in the proof of Theorem~\ref{thm:main formula} give the result.
\end{proof}

\begin{rmk}
The 
right hand side of the formula (\ref{form:red})
can be obtained by Kawai-Yoshioka's formula \eqref{equ:KW}.
\end{rmk}

\section{Stable pairs on nodal $K3$ fibrations}\label{sec:nodalK3}
\subsection{Nodal $K3$ fibrations} \label{sec:conifold}
In this section,
 we aim to prove the compatibility condition for stable
pair invariants via conifold transitions. 
Let $X$ be a smooth projective 3-fold, and 
$\pi$ be a $K3$ fibration
\begin{align*}
\pi:X\rightarrow C
\end{align*}
onto a smooth projective curve $C$. 
A $K3$ fibration $\pi$ is called a \textit{nodal K3 fibration}
if the singularities of the fibers of $\pi$
are at worst 
ordinary double point (ODP)
singularities. 
\begin{exam}\label{exam:nodal}
Let $X \subset \mathbb{P}^3 \times \mathbb{P}^1$
be a generic hyperplane section of bidegree $(4, 2)$, and 
$\pi$ the composition
\begin{align*}
\pi \colon X \subset \mathbb{P}^3 \times \mathbb{P}^1 \to \mathbb{P}^1
\end{align*}
where the second morphism is the projection. Then $\pi$
is a nodal $K3$ fibration. 
\end{exam}
For any $\beta \in H_2(X,\Z)^\pi$, the stable pair invariants $P_{n,\beta}$ are defined as before by
\begin{align*}
P_{n, \beta} = 
\int_{[\pP_n(X, \beta)]^{\rm{\rm{vir}}}} 1,
\end{align*} where ${[\pP_n(X, \beta)]^{\rm{\rm{vir}}}}$ is the 0-dimensional virtual cycle associated to the obstruction theory in Theorem \ref{PT}.
We will study the generating series $\PT(X)_{\beta}$, 
$\PT(X)$ given by (\ref{gen:PT}). 
 \begin{rmk}
When the singularities of fibers of $\pi$
are more general type of rational double points (RDP), 
then the study of stable pair invariants may be reduced to 
the nodal case by using the deformation invariance of 
stable pair invariants.   
\end{rmk}
\subsection{Conifold transition}
Let 
\begin{align*}s_1,\dots, s_k\in X, \ 
c_1, \cdots, c_{k'} \in C
\end{align*}
are the singular points of the fibers of $\pi \colon X \to C$, 
the points in $C$ over which the fibers have singularities, 
respectively.  
By our assumption, 
each $s_i \in X$ is an ordinary double point (nodal) singularity for any $i$
in the fiber of $\pi$. 
If $k'$ is even, define $c_0=c_1$ and if $k'$ is odd, define $c_0$ to be an arbitrary point of $C$ distinct from $c_1,\dots,c_{k'}$. 
Let
\begin{align*}
\epsilon:\tC\to C
\end{align*}
be the double cover of $C$ branched over the points $c_0, c_1, \dots,c_{k'}$. 
It can be seen that $X_0 \cneq 
\epsilon^*(X)$ is a threefold with the conifold singularities. 
Let
$h \colon \tX \to X_0$
be its small resolution with the exceptional nonsingular rational curves 
\begin{align}\label{exc}
e_1, \cdots, e_k \subset \tX, \ 
h(e_i)=s_i
\end{align}
and $\tpi:\tX\to \tC$ the induced morphism. 
In general, the small resolution $\widetilde{X}$ may 
not be a projective variety, 
but is realized as an algebraic space by~\cite{Artin}. 
As a summary, we have the commutative diagram
\begin{align}\label{Cartesian}
\xymatrix{
\widetilde{X} \ar[r]^{h} \ar[rd]_{\widetilde{\pi}}
\ar@/^18pt/[rr]^{\widetilde{\epsilon}}
 & X_0 \ar[d] \ar[r] \ar@{}[dr]|\square & X \ar[d]^{\pi} \\
&   \widetilde{C}  \ar[r]^{\epsilon}  &   C. 
}
\end{align}
The normal bundle of $e_i$ in $\tX$ is isomorphic to $\oO_{\mathbb{P}^1}(-1)\oplus \oO_{\mathbb{P}^1}(-1)$ \cite{a121}. Moreover, let $\epsilon_t: \tC_t\to C$ be a double cover of $C$ branched at $k+2\{k/2\} $ generic points\footnote{$\{k/2\}=0$ if $k$ is even, and $\{k/2\}=1/2$ if $k$ is odd.} of $C$ when $t\neq 0$, and set $\tC_0=\tC$. Define $X_t=\epsilon^*_t(X)$, 
so we have the Cartesian square
\begin{align*}
\xymatrix{
 X_t \ar[d] \ar[r] \ar@{}[dr]|\square
& X \ar[d]^{\pi}  \\
   \widetilde{C}_t  \ar[r]^{\epsilon_t}  &   C. 
}
\end{align*}
 Our plan is to relate stable pair 
 invariants of $\tX$ and $X_t$ which differ by the conifold transitions. As in GW theory~\cite{a96}, \cite{a97}, this can be done using degeneration techniques.

\begin{rmk}\label{rem:nonproj}
We will not pursue the 
foundation of moduli theory of stable pair invariants 
on the algebraic space $\widetilde{X}$. 
The Hilbert schemes of curves 
 on algebraic spaces are realized as algebraic 
spaces~\cite[Corollary~6.2]{Artin2}, 
and the similar argument may be applied 
to construct
 the moduli spaces of stable pairs on 
$\widetilde{X}$ as algebraic spaces. 
Also 
in the situation that any stable pair 
on $\widetilde{X}$ is scheme theoretically 
supported on a $K3$ fiber,
then 
 by Proposition~\ref{pushforward}, the moduli space of stable pairs 
can alternatively be constructed as a relative Hilbert scheme of points over 
the universal curve of the relative linear system.
Hence it is an algebraic space by~\cite[Corollary~6.2]{Artin2}.
\end{rmk}

\begin{rmk} \label{alg-space-degen}
If the moduli space of 
stable pairs on $\tX$ is an algebraic space, then 
the arguments in~\cite{BF}
show the existence of the
zero dimensional
 virtual fundamental 
class. 
Below 
we will also assume that 
Li-Wu's degeneration formula~\cite{a87} 
works for 
stable pair invariants 
on algebraic spaces. 
Of course, we don't have to 
address the foundational matters in Remark~\ref{rem:nonproj}
and this remark
if we can choose the small resolution $\widetilde{X}$ to be
projective.   
\end{rmk}

\subsection{Relative stable pair invariants}
In order to apply the degeneration technique, we 
will use the notion of relative stable pair theory
given by Li-Wu~\cite{a87}. 
\begin{defi}
Let $W$ be a smooth projective threefold and $D\subset W$ a smooth divisor. 
For $\beta \in H_2(W, \mathbb{Z})$ and $n\in \mathbb{Z}$, we denote 
by  
$\pP_n(W/D, \beta)$ the moduli stack of relative stable pairs 
$(\oO_W \to F)$ on $W$
satisfying 
$[F]=\beta$ and $\chi(F)=n$, 
in the sense of~\cite{a87}. 
\end{defi}
There is an open substack of
$\pP_n(W/D, \beta)$ whose $\CC$-points 
correspond to pairs $(\oO_W \to F)$
such that $F$ is supported on $W$ and normal to $D$, 
i.e. $\tT or_1^{\oO_X}(F, \oO_D)=0$. 
A $\CC$-point of the boundary corresponds to an admissible 
stable pair
(cf.~\cite[Definition~4.8]{a87})
supported on an $n$-step degeneration of $(W[n], D[n])$. 
By the relativity of stable pairs, 
the restriction map defines the morphism
\begin{align*}
\mathrm{ev} \colon 
\pP_n(W/D, \beta) \to \Hilb(D, \lvert \eta \rvert). 
\end{align*}
Given a cohomology weighted partition 
$\eta$ with respect to 
a basis of $H^{\ast}(D, \mathbb{Q})$, we can associate 
a cohomology class 
\begin{align*}
C_{\eta} \in H^{\ast}(\Hilb(D, \lvert \eta \rvert), \mathbb{Q})
\end{align*} 
which forms a basis of $H^{\ast}(\Hilb(D, \lvert \eta \rvert), \mathbb{Q})$
called \textit{Nakajima basis} (cf.~\cite{a127}). 
If we choose a basis of $H^{\ast}(D, \mathbb{Q})$ which is self-dual 
with respect to the Poincar\'e pairing, 
then for each cohomology weighted partition 
$\eta$, there is a dual partition $\eta^{\vee}$ such that
\begin{align*}
\int_{\Hilb(D, \lvert \eta \rvert)}
C_{\eta} \cup C_{\nu} = \frac{(-1)^{\lvert \eta \rvert-l(\eta)}}{a(\eta)}
\delta_{\nu, \eta^{\vee}}
\end{align*}
for any cohomology weighted partition $\nu$ with 
$\lvert \nu \rvert = \lvert \eta \rvert$. 
Here $l(\eta)$ is the length of the partition $\eta$
and $a(\eta)$ is defined by 
\begin{align*}
a(\eta)= \prod_{i} \eta_i \lvert \Aut(\eta) \rvert. 
\end{align*}
\begin{defi}
The relative stable pair invariant (without insertions) is defined by 
\begin{align*}
P_{n, \beta}(W/D)_{\eta} \cneq 
\int_{[\pP_n(W/D, \beta)]^{\rm{\rm{vir}}}} \mathrm{ev}^{\ast} C_{\eta}. 
\end{align*}
\end{defi}
The virtual dimension of $P_n(W/D, \beta)$
is given by $c_1(W) \cdot \beta$. Therefore the above invariant 
is zero unless
\begin{align*}
c_1(W) \cdot \beta = \deg C_{\eta} \ge 0. 
\end{align*}
We define the following generating series
\begin{align*}
\PT(W/D)_{\beta, \eta} &= \sum_{n} P_{n, \beta}(W/D)_{\eta}q^n \\
\PT(W/D)_{\eta} &= \sum_{\beta} \PT(W/D)_{\beta, \eta}t^{\beta}. 
\end{align*}
We drop $\eta$ or $D$ from the notation if respectively $\lvert \eta \vert =0$ or $D=\emptyset$.

\subsection{Degeneration formula}
Let $Y$ be the 
threefold obtained by blowing up $X_0$ at 
all the conifold points, with the exceptional divisors
\begin{align*}
D\cneq 
D_1 \sqcup \dots \sqcup D_k, \ 
D_i=\mathbb{P}^1 \times \mathbb{P}^1.
\end{align*}
Note that we have the factorization
\begin{align*}
f \colon Y \stackrel{g}{\to} \tX \stackrel{h}{\to} X_0
\end{align*}
such that $g$ is the blowing up at all the exceptional loci (\ref{exc})
of 
$h$. 
 We use degenerations of the threefolds $\tX$ and $X_t$ to respectively 
\begin{align}\label{degen}
Y \bigcup_{D_1,\dots,D_k} \coprod_{i=1}^k \bP_{1} \quad \text{and}  \quad Y \bigcup_{D_1,\dots,D_k} \coprod_{i=1}^k \bP_{2}. 
\end{align} 
Here $\bP_1$ is 
given by 
$$\bP_{1}\cong \P(\O_{\P^1}\oplus \O_{\P^1}(-1) \oplus \oO_{\P^1}(-1))$$ and 
$\bP_{2}$ is a smooth quadric hypersurface in $\P^4$.

The first degeneration in (\ref{degen}) is the degeneration to 
the normal cone \cite{a35} in which $D_i \subset Y$ is attached to the divisor at infinity $H_1=\P(\O_{\P^1}(-1)^2)$ in the $i$-th copy of $\bP_1$. The second degeneration is called the semi-stable reduction of a conifold degeneration 
\cite{a96} in which $D_i \subset Y$ 
is attached to a smooth hyperplane section $H_2$ in the $i$-th copy of $\bP_2$.
We denote by 
\begin{align*}
\X_1 \to \mathbb{A}^1, \quad \X_2\to \mathbb{A}^1
\end{align*} the total spaces of the first and second degenerations above. 
Let $L$ be an ample line bundle on $X$. 
We define line bundles $L_t$ on $X_t$ 
to be 
\begin{align*}
L_t= \epsilon^{\ast}_t(L)
\end{align*}
where we set $\epsilon_0 \cneq \epsilon$.

Now we apply the degeneration formula of 
stable pair invariants with respect to the above degenerations. 
Suppose for simplicity
that there is only one critical locus 
for the fibration 
$X \to C$, i.e. $k=1$. 
 By the degeneration formula, 
we obtain the following identities
for $\beta \in H_2(\tX, \mathbb{Z})^{\tpi}$
\begin{align}\label{deg:1}
&\PT(\tX)_{\beta}= \\
\notag
&\sum_{\eta, \beta_1 + \beta_2=\beta} 
\PT(Y/D_1)_{\beta_1, \eta} 
\frac{(-1)^{\lvert \eta \rvert -l(\eta)}a(\eta)}{q^{\lvert{\eta}\rvert}}\PT(\bP_1/H_1)_{\beta_2, \eta^{\vee}}
\end{align}
and 
\begin{align}\label{deg:2}
&\PT(X_t)_{\beta}=\\
\notag
&\sum_{\eta, \beta_1 + \beta_2=\beta} 
\PT(Y/D_1)_{\beta_1, \eta} 
\frac{(-1)^{\lvert \eta \rvert -l(\eta)}a(\eta)}{q^{\lvert{\eta}\rvert}}\PT(\bP_2/H_2)_{\beta_2, \eta^{\vee}}. 
\end{align}
Here the sum $\beta_1 + \beta_2 =\beta$ is an equality in 
$H_2(\X_1)$ and $H_2(\X_2)$ respectively.


\begin{prop}\label{prop:deg1}
The degeneration formula (\ref{deg:1}) implies that
\begin{align*}
\PT(\tX)_{\beta}= 
\sum_{\begin{subarray}{c} 
\beta_1 \in H_2(X_0), \ \beta_2 \in H_2(\tX) \\
h^{!}\beta_1 + \beta_2=\beta, \ h_{\ast}\beta_2=0
\end{subarray}}
\PT(Y/D)_{f^{!}\beta_1} \cdot \PT(\tX)_{\beta_2}. 
\end{align*}
\end{prop}
\begin{proof}
The degeneration formula of relative rank one DT invariants 
for the blow-up at $(-1, -1)$-curves 
is worked out by Hu-Li~\cite{Hu-Li}, and we apply the same argument. 
By the agreement of the virtual dimensions, it 
is proved in~\cite[Theorem~4.2]{Hu-Li}
that a non-zero term of the RHS of (\ref{deg:1})
satisfies
\begin{align*}
0=
\lvert \eta \vert  = \beta_1 \cdot D_1 = \beta_2 \cdot H_1. 
\end{align*}
This implies that $\beta_1$ is written as $h^{!}\beta_1'$ for 
some $\beta_1' \in H_2(X_0)$, and $\beta_2$ is a multiple of 
the class of the curve $e \subset \bP_1$
given by the embedding 
\begin{align*}
\oO_{\bP_1} \subset \oO_{\bP_1} \oplus \oO_{\bP_1}(-1) \oplus \oO_{\bP_1}(-1)
\end{align*}
into the first factor. 
The curve $e$ is a $(-1, -1)$-curve which does not intersect with $H_1$. 
The contributions of the relative stable pairs on $\bP_1$
with curve class $m[e]$ is 
identified with the stable pair invariants on $\tX$ with curve class 
$m[e_1]$\footnote{This follows by a parallel argument as (4.4) in the proof of \cite[Theorem~4.2]{Hu-Li}.}. Hence we obtain the desired result
for $k=1$. The case of $k>1$ is similarly discussed. 
\end{proof}

\begin{prop}\label{prop:deg2}
For any $d\in \mathbb{Z}_{>0}$, the 
degeneration formula (\ref{deg:2}) implies that
\begin{align*}
\sum_{\begin{subarray}{c}
\beta \in H_2(X_t) \\
 L_t \cdot \beta=d
\end{subarray}}
\PT(X_t)_{\beta}=
\sum_{\begin{subarray}{c}
\beta' \in H_2(X_0) \\
L_0 \cdot \beta'=d
\end{subarray}}
\PT(Y/D)_{f^{!}\beta'}. 
\end{align*}
\end{prop}
\begin{proof}
By the agreement of the virtual dimensions, 
a non-zero term of the RHS of (\ref{deg:2}) satisfies
\begin{align*}
c_1(Y) \cdot \beta_1= -D_1 \cdot \beta_1 = \deg C_{\eta} \ge 0. 
\end{align*}
We also have the following compatibility condition
\begin{align*}
D_1 \cdot \beta_1= H_2 \cdot \beta_2 = \lvert \eta \rvert \ge 0. 
\end{align*}
The above inequalities imply that 
$D_1 \cdot \beta_1= \lvert \eta \rvert =0$
and $H_2 \cdot \beta_2=0$. 
The first equality implies that
$\beta_1$ is written as $f^{!}\beta'$ for some
$\beta' \in H_2(X_0)$. 
The second equality implies that
$\beta_2=0$, 
since $H_2$ is an ample divisor
in $\bP_2$.  
By rearranging the sum, we obtain the desired formula
for $k=1$. The case of $k>1$ is similarly discussed. 
\end{proof}

\subsection{Relation between $\PT(X)$ and $\PT(\tX)$}
Now we choose $k'$ generic fibers $S_1,\dots, S_{k'}$ of $X\to C$. 
By our assumption $S_i$ is a $K3$ surface. 
Let $X_i=S_i\times \P^1$. Then $X_i$ is a smooth $K3$-fibration over $\P^1$. 
We identify the surface $S_i$ with the divisor $S_i \times \{0\}$ in $X_i$. 

\begin{lem} \label{K3-vanishing}
For any curve class $\beta$
contained in fibers of $X_i \to \P^1$, we have 
$P_{n, \beta}(X_i/S_i)=0$.
\end{lem}

\begin{proof}
By (21) in \cite{MPT} the cup product map, $\cup \beta: H^1(T_{S_i})\to H^2(\O_{S_i})$ is surjective, and hence by the proof of \cite[Theorem 2.7]{a112} the obstruction theory of $\pP_n(X_i/S_i,\beta)$ contains a trivial factor which implies the vanishing of the invariants\footnote{The proof of \cite[Theorem 2.7]{a112} is given for absolute geometries but a parallel argument applies to the relative geometry here.}.
\end{proof}
 
We set $S$ to be the disjoint union of 
$S_i$ for $1\le i\le k'$. 
\begin{lem} \label{lem:rel=abs}
$P_{n, \beta}(X/S)=P_{n, \beta}(X)$.
\end{lem}
\begin{proof}
Use the degeneration formula for the degeneration of $X$ into $$X \bigcup_{S_1,\dots, S_r} \coprod_i X_i.$$ The
 vanishing of $P_{n, \beta}(X_i/S_i)$ from Lemma \ref{K3-vanishing} proves the claim. 
\end{proof}

We use Lemma \ref{lem:rel=abs} to relate the PT invariants of $X$ to $X_t$.  To achieve this we use the degeneration of $X_t$ obtained by degenerating its base $\tC$ to two copies of $\C$ by attaching two copies of $X$ along the generic fibers $S_1,\dots,S_{k'}$.  The degeneration formula then implies 
that
\begin{align}\label{deg:3}
\sum_{L_t \cdot \beta=d}
\PT(X_t)_{\beta}=
\sum_{d_1 + d_2=d} 
\sum_{L \cdot \beta_1 =d_1} \PT(X/S)_{\beta_1} \sum_{L \cdot \beta_2=d_2}
\PT(X/S)_{\beta_2}
\end{align}
for each $d \in \mathbb{Z}_{>0}$. 
Let $\PT^{h}(\tX)$ be the generating 
series defined by 
\begin{align} \label{equ:contracted}
\PT^{h}(\tX) &\cneq 
\sum_{h_{\ast}\beta=0} P_{n, \beta}(\tX) q^n t^{\beta} \\ \notag
&=\prod_{1\le i\le k, n\ge 1} (1-(-q)^n t^{[e_i]})^n. 
\end{align}
Here the second identity follows from the 
computation of stable pair invariants on 
a $(-1, -1)$-curve~\cite{NN}. 
Let $\widetilde{\epsilon} \colon \tX \to X$ be the natural 
morphism
given in the diagram (\ref{Cartesian}). 
Combined with the results in the previous subsection, we obtain the 
following result:
\begin{thm}\label{thm:formula:relation}
We have the formula
\begin{align}\label{formula:relation}
\widetilde{\epsilon}_{\ast} \frac{\PT(\tX)}{\PT^{h}(\tX)}
= \PT(X)^2. 
\end{align}
Here $\widetilde{\epsilon}_{\ast}$
is the variable change $t^{\beta} \mapsto 
t^{\widetilde{\epsilon}_{\ast}\beta}$. 
\end{thm}
\begin{proof}
For each $d\in \mathbb{Z}_{>0}$, 
Proposition~\ref{prop:deg2}, Lemma~\ref{lem:rel=abs} and (\ref{deg:3}) imply
\begin{align*}
\sum_{\begin{subarray}{c}
\beta \in H_2(X_0) \\ L \cdot \epsilon_{\ast} \beta =d
\end{subarray}}
\PT(Y/D)_{f^{!}\beta}
= \sum_{\begin{subarray}{c}
\beta_1, \beta_2 \in H_2(X) \\
L \cdot \beta_1 + L \cdot \beta_2 =d
\end{subarray}}
\PT(X)_{\beta_1} \cdot \PT(X)_{\beta_2}. 
\end{align*}
Since the above formula holds for any $L$, we have 
\begin{align}\label{deg:4}
\widetilde{\epsilon}_{\ast}
\sum_{\beta \in H_2(X_0)}
\PT(Y/D)_{f^{!}\beta} t^{h^{!}\beta}
= \PT(X)^2. 
\end{align}
On the other hand, the LHS of (\ref{deg:4})
coincides with the LHS of (\ref{formula:relation}) by 
Proposition~\ref{prop:deg1}.
Therefore we obtain the result.   
\end{proof}
\begin{rmk}
The variable change $t^{\beta} \mapsto 
t^{\widetilde{\epsilon}_{\ast}\beta}$
only makes sense after taking the quotient series in 
the LHS of (\ref{formula:relation}). 
Otherwise there are infinite numbers of contributions 
to the coefficients. 
\end{rmk}

\begin{rmk}
The LHS of (\ref{formula:relation}) is 
known to be independent of a choice of a small 
resolution~\cite{Hu-Li},~\cite{Tcurve2},~\cite{Cala}. 
The formula (\ref{formula:relation}) is consistent with 
this fact. 
\end{rmk}
In an irreducible curve class case, 
we have the following corollary: 
\begin{cor}\label{cor:irred}
For an irreducible curve class $\beta \in H_2(X, \mathbb{Z})^{\pi}$, 
we have the following formula: 
\begin{align}\label{irred:formula}
\sum_{\widetilde{\epsilon}_{\ast}\widetilde{\beta}=\beta}
\widetilde{\epsilon}_{\ast}
\frac{\sum_{h_{\ast}e=0}
\PT(\widetilde{X})_{\widetilde{\beta}+e}}{\PT^h(\widetilde{X})}
=2 \PT(X)_{\beta}. 
\end{align}
Here $\widetilde{\beta} \in H_2(\widetilde{X}, \mathbb{Z})$
are irreducible curve classes, 
giving lifts of $\beta$. 
\end{cor}

\subsection{Conjectural relation to perverse (non-commutative) stable pair theory} \label{sec:noncomm}
In this subsection, 
we introduce perverse (or non-commutative) stable 
pair theory, and propose a conjectural 
relationship between the quotient 
series $\PT(\tX)/\PT^h(\tX)$ in Theorem~\ref{thm:formula:relation}
and the generating series of the perverse (non-commutative)
stable pair invariants. 
We first recall the heart of the perverse t-structure 
on $D^b \Coh(\tX)$
associated to the small resolution 
$h \colon \tX \to X_0$, 
introduced by Bridgeland~\cite{Br1}. 
Let $\cC \subset D^b \Coh(\widetilde{X})$ be the full 
subcategory defined by 
\begin{align*}
\cC = \{ E \in D^b \Coh(\tX) : \dR h_{\ast} E =0\}. 
\end{align*}
By~\cite{Br1}, the standard t-structure on $D^b \Coh(\tX)$
induces a t-structure $(\cC^{\le 0}, \cC^{\ge 0})$
on $\cC$. 
We define $\Per(\tX/X_0)$ to be 
\begin{align*}
\Per(\tX/X_0) = \left\{ E \in D^b \Coh(\tX) :
\begin{array}{c} 
\dR h_{\ast} E \in \Coh(X_0), \\
\Hom(\cC^{<0}, E)= \Hom(E, \cC^{>0})=0 
\end{array}
\right\}. 
\end{align*}
The category $\Per(\tX/X_0)$ is the heart 
of a bounded t-structure on $D^b \Coh(\tX)$, hence it is 
an abelian category. 
Note that $\oO_X \in \Per(\tX/X_0)$ by definition. 
Let $\Per_{\le 1}(\tX/X_0)$, 
$\Per_0(\tX/X_0)$ be the subcategories of 
$\Per(\tX/X_0)$, defined by 
\begin{align*}
&\Per_{\le 1}(\tX/X_0)
= \{ E \in \Per(\tX/X_0): 
\dim h(\Supp(E)) \le 1 \}, \\
&\Per_0(\tX/X_0)
= \{ E \in \Per(\tX/X_0): 
\dim h(\Supp(E))=0\}. 
\end{align*}
It is easy to see that 
$\Per_{\le 1}(\tX/X_0)$ and 
$\Per_0(\tX/X_0)$ are
closed under quotients and subobjects in 
$\Per(\tX/X_0)$. In particular, they
are abelian subcategories of $\Per(\tX/X_0)$. 
\begin{defi}\label{defi:pstable}
A perverse stable pair for $h \colon \tX \to X_0$
consists of data
$(F, s)$
\begin{align}\label{pstable}
F \in \Per_{\le 1}(\tX/X_0), \quad  s \colon \oO_{\tX} \to F
\end{align}
satisfying the following conditions: 
\begin{itemize}
\item $\Hom(\Per_0(\tX/X_0), F)=0$. 
\item The cokernel of $s$ in $\Per(\tX/X_0)$
is an object in $\Per_0(\tX/X_0)$. 
\end{itemize}
\end{defi}
The above definition coincides with the usual definition 
of stable pairs if $h \colon \tX \to X_0$ is an isomorphism
(i.e. there is no nodal fiber for $\pi \colon X \to C$). 
Also the above perverse stable pairs are related to 
the non-commutative version of stable pairs as follows. 
Suppose that $\tX$ were a projective variety. 
Then, Van den Bergh~\cite{MVB} shows that there is a vector bundle 
$\eE_0$ on $\tX$ such that 
$\eE=\oO_{\tX} \oplus \eE_0$ gives a derived equivalence
\begin{align*}
\Phi= \dR h_{\ast} \dR \hH om(\eE, \ast) 
 \colon D^b \Coh(\tX) \stackrel{\sim}{\to}
D^b \Coh(\aA_{X_0})
\end{align*}
where $\aA_{X_0}=h_{\ast} \eE nd(\eE)$ is the sheaf of non-commutative
algebras on $X_0$. 
The equivalence $\Phi$ restricts to an equivalence between 
$\Per(\tX/X_0)$ and $\Coh(\aA_{X_0})$. 
Let $\uU=\Phi(\oO_{\tX}) \in \Coh(\aA_{X_0})$ 
be the local projective generator of $\aA_{X_0}$. 
Let 
\begin{align*}
\Coh_{\le 1}(\aA_{X_0}), \quad 
\Coh_0(\aA_{X_0})
\end{align*}
be the subcategories of $\Coh(\aA_{X_0})$, 
consisting of $E \in \Coh(\aA_{X_0})$
whose support as $\oO_{X_0}$-module 
has dimension $\le 1$, $0$ respectively. 
It is easy to see that
\begin{align*}
&\Phi(\Per_{\le 1}(\tX/X_0)) = \Coh_{\le 1}(\aA_{X_0}), \\
& \Phi(\Per_0(\tX/X_0)) = \Coh_0(\aA_{X_0}). 
\end{align*}
Under the equivalence $\Phi$, the data (\ref{pstable}) is equivalent to 
the data $(F', s')$
\begin{align}\label{npair}
F' \in \Coh_{\le 1}(\aA_{X_0}), \quad s' \colon \uU \to F'
\end{align}
such that 
$\Hom(\Coh_0(\aA_{X_0}), F')=0$ and 
the cokernel of $s$ in $\Coh(\aA_{X_0})$ is 
an object in $\Coh_0(\aA_{X_0})$. 

The pair (\ref{npair}) is a non-commutative analogue of 
stable pairs. However unfortunately, our 3-fold $\tX$ 
may 
not be projective, and the 
sheaf of non-commutative algebras $\aA_{X_0}$ only exists at 
formal neighborhoods at each point in $X_0$. 
Because of the absence of the global 
sheaf of non-commutative algebras $\aA_{X_0}$, we 
formulate our non-commutative stable pairs as
a morphism in the perverse heart as in (\ref{pstable}). 

For $\beta \in H_2(\tX, \mathbb{Z})^{\widetilde{\pi}}$ 
and $n\in \mathbb{Z}$, let
\begin{align}\label{moduli}
\pP_n^{\rm{per}}(\tX, \beta)
\end{align}
be the moduli space of perverse stable pairs 
$(\oO_{\tX} \to F)$ with $[F]=\beta$ and $\chi(F)=n$. 
We expect
 that (see~Remark~\ref{rmk:permoduli})
the moduli space (\ref{moduli}) 
exists as a proper algebraic space, 
with a perfect obstruction theory with virtual 
dimension zero given by 
\begin{align}\label{per-PT}
\eE^{\bullet}_{\rm{per}}:= \dR \hH om_{\pi_{\pP}}
(\Ibb_{\rm{per}}, \Ibb_{\rm{per}}\otimes \omega_{\pi})_{0}[2]
\to \mathbb{L}_{\pP_n^{\rm{per}}(\tX, \beta)}^{\bullet}.
\end{align}
Here $\Ibb_{\rm{per}}$ is the total complex associated to 
the universal morphism 
\begin{align*}
s \colon \oO_{\tX\times \pP^{\rm{per}}_{n}(\tX,\beta)} \to
 \mathbb{F}_{\rm{per}}.
\end{align*}

\begin{rmk}\label{rmk:permoduli}
Here we assumed that (\ref{moduli})
is a proper algebraic space with perfect obstruction theory (\ref{per-PT}). 
We have not checked all the details but we believe that they should all follow from the 
similar arguments in the existing literatures. 
First the argument in~\cite[Section~6.4, Step~2]{Tcurve2}
should show that (\ref{moduli}) is an algebraic space of finite type.
The properness should follow from the arguments in~\cite{JLo2}.  
Also the existence of the perfect 
obstruction theory (\ref{per-PT}) should follow from
the argument in~\cite{a10} by noting that the complexes associated to perverse
stable pairs in Definition~\ref{defi:pstable} are simple objects of the derived category as they are stable with respect to 
some weak stability condition (see Remark \ref{conjequal}).
\end{rmk}

By taking the integration, we obtain the invariant
\begin{align*}
P_{n, \beta}^{\rm{per}} = 
\int_{[\pP_n^{\rm{per}}(\tX, \beta)]^{\rm{\rm{vir}}}} 1. 
\end{align*}
We define the generating series 
$\PT^{\rm{per}}(\tX)_{\beta}$ and
$\PT^{\rm{per}}(\tX)$ to be
\begin{align*}
\PT^{\rm{per}}(\tX)_{\beta}=
\sum_{n \in \mathbb{Z}} P_{n, \beta}^{\rm{per}}q^n, \ 
\PT^{\rm{per}}(\tX)=
\sum_{\beta \in H_2(\tX, \mathbb{Z})^{\tpi}} 
\PT(\tX)_{\beta}t^{\beta}.  
\end{align*}

\begin{conj}\label{conj:PT=perv}
We have the equality
\begin{align}\label{PT=perv}
\frac{\PT(\tX)}{\PT^h(\tX)}
= \PT^{\rm{per}}(\tX). 
\end{align}
\end{conj}

\begin{rmk} \label{conjequal}
The conjectural equality (\ref{PT=perv}) is motivated by
the third author's 
work~\cite{Tcurve2}. 
Indeed the equality (\ref{PT=perv})
holds true for the Euler characteristic version, i.e.
replace all the invariants by the topological Euler characteristic 
of the moduli spaces. 
This follows immediately from the proof of~\cite[Theorem~5.8]{Tcurve2}, 
and the fact that $\sigma_{\xi(0)}$-semistable objects
in the notation of~\cite[Theorem~5.8]{Tcurve2}
consist of total complexes associated to perverse
stable pairs in Definition~\ref{defi:pstable}.
(The latter fact can be easily proved using the same 
argument of~\cite[Proposition~5.5]{Tcurve2}).
Moreover, if our 3-fold $\tX$ were projective 
Calabi-Yau 3-fold, then the equality (\ref{PT=perv}) 
follows from the results of~\cite{Tcurve2}, \cite{Cala}. 
However, since our 3-fold $\tX$ is no longer 
Calabi-Yau, we need a new technique of wall-crossing to 
show the inequality (\ref{PT=perv}). 
\end{rmk}

\begin{rmk}
If Conjecture~\ref{conj:PT=perv} is true, 
then Theorem~\ref{thm:formula:relation}
implies that
\begin{align}\label{PT=perv2}
\widetilde{\epsilon}_{\ast} \PT^{\rm{per}}(\tX)=\PT(X)^2. 
\end{align}
Hence we can reduce the computation of $\PT(X)$
for a nodal $K3$ fibration $X \to C$
to the computation of $\PT^{\rm{per}}(\tX)$
for a smooth $K3$ fibration $\tX \to \tC$. 
\end{rmk}

\subsection{Computation of perverse stable pair invariants}
Finally in this section, we aim at
computing perverse stable pair invariants. 
We prove the following lemma: 
\begin{lem}\label{stable-perv}
Let $(\oO_{\widetilde{X}} \to F)$ be a perverse stable pair 
for $\tX \to X_0$
such that $h_{\ast}[F]$ is an irreducible curve class in $X_0$. 
Then we have
\begin{enumerate}
\item $\Hom(F, F)=\mathbb{C}$.
\item $\operatorname{Supp}(F) \subset S$ where $i:S\hookrightarrow \tX$ is a K3 fiber and $\operatorname{Supp}(F)$ is the scheme theoretic support of the complex $F$. 
\end{enumerate}  
\end{lem}
\begin{proof} (1)
It is enough to show that 
any non-zero morphism 
$u \colon F \to F$ in $\Per(\tX/X_0)$
is an isomorphism. 
Let us take a factorization in $\Per(\tX/X_0)$
\begin{align*}
F \twoheadrightarrow \Imm u \hookrightarrow F. 
\end{align*}
Let $U=\tX \setminus \mathrm{Ex}(h)$.
Suppose first that $u|_{U} =0$. Then 
$\imm u$ is a non-zero object in 
$\Per_0(\tX/X_0)$, which contradicts to 
Definition~\ref{defi:pstable}.
Hence we have $u|_{U} \neq 0$.  
Since the object $F|_{U}$ is a pure one dimensional 
sheaf whose scheme theoretic support is
irreducible, the morphism $u|_{U}$ is an injective 
morphism in $\Coh(U)$. This implies that, 
if $\Ker(u) \in \Per(\tX/X_0)$
is non-zero, then it is an object in $\Per_0(\tX/X_0)$. 
This contradicts to Definition~\ref{defi:pstable}, 
hence $\Ker(u)=0$. 

Now we have proved that $u$ is injective in $\Per(\tX/X_0)$. 
It remains to show that $\Cok(u) \in \Per(\tX/X_0)$
vanishes. 
Since $\dR h_{\ast}$ takes 
$\Per(\tX/X_0)$ to $\Coh(X_0)$, we have the 
exact sequence in $\Coh(X_0)$
\begin{align*}
0 \to \dR h_{\ast} F \to \dR h_{\ast} F \to \dR h_{\ast} \Cok(u) \to 0. 
\end{align*}
The above sequence implies that $\dR h_{\ast} \Cok(u)=0$, 
hence $\Cok(u) \in \cC$.
It follows that, by~\cite[Proposition~3.5.8]{MVB},
the object $\Cok(u)$ is isomorphic to the direct sum of 
the sheaves
$\oO_{e_i}(-1)$.
Therefore 
by taking the cohomology long exact sequence 
associated to $0 \to F \to F \to \Cok(u) \to 0$, and noting that $\hH^i(F)=0$ for $i\neq 0,-1$ (because $F\in \Per(\tX/X_0)$),
we obtain the exact sequence of sheaves 
\begin{align*}
0 \to \hH^0(F) \to \hH^0(F) \to \Cok(u) \to 0. 
\end{align*}  
Therefore we have $\Cok(u)=0$. 

(2) By part (1), $F$ is necessarily supported (set theoretically) on only one fiber $S$. Tensoring the natural short exact sequence $0\to \O_{\tX}(-S)\to \O_{\tX}\to i_*\O_S\to 0$ by $F$, and noting that $F\otimes \O_{\tX}(-S)\cong F$, we get the exact triangle $F\to F\to  i_*\dL i^*F$. Since $\dL i^*F\not \cong 0$, the first map in the exact triangle cannot be an isomorphism, and hence it must be zero map by part (1) of the lemma. The exact triangle then implies that $ i_*\dL i^*F\cong F\oplus F[1]$, and so in particular $\operatorname{Supp}(F)=\operatorname{Supp}(\dL i^*F)$. Part (2) now follows from \cite[Lemma 3.29]{FMAG}. 
\end{proof}

For the smooth $K3$ fibration $\tpi:\tX\to \tC$, 
let us take a curve class 
\begin{align} \label{equ:classs2} 
\overline{\beta}=\widetilde{\beta}+me_i\in H_2(\tX,\Z)^{\tpi},
\end{align} 
for some irreducible curve class $\widetilde{\beta}
\in H_2(\widetilde{X}, \Z)$ 
and $m\in \Z_{\ge 0}$. 
Since $\widetilde{\beta}$ is an irreducible class Lemma \ref{stable-perv} applies for the perverse stable pairs $(\oO_{\widetilde{X}} \to F)$ such that $[F]=\overline{\beta}$. Motivated by the result of this lemma we assume the following conjecture:\footnote{This is an analog of Lemma \ref{lem:point}.}
\begin{conj}\label{conj:support} Let $(\oO_{\widetilde{X}} \to F)$ be a perverse stable pair 
for $\tX \to X_0$ such that $[F]=\overline{\beta}$. Using the notation of  Lemma \ref{stable-perv} Part (2), there exists a complex $G$ on $S$ such that $F=i_*G$.
\end{conj}

Assuming Conjecture \ref{conj:support}, there is a natural morphism 
\begin{align*}
\rho \colon \pP_n^{\rm{per}}(\tX, \overline{\beta})\to \tC. 
\end{align*}
Also similarly to Definition~\ref{defn:comps}, 
we can define the type I and  type II components $\pP^{\rm{per}}_{\rm{I}}, \pP^{\rm{per}}_{\rm{II}}$ of 
$\pP_n^{\rm{per}}(\tX, \overline{\beta})$. For a point $p\in \tC$, suppose that $S:=\tpi^{-1}(p)\xhookrightarrow{i} \tX$ and $\gamma$ is a curve class on $S$ with $i_*\gamma=\overline{\beta}$, then we denote the fiber $\rho^{-1}(p)$ of $\rho$ restricted to a type I or an isolated type II component by $\pP_n^{\rm{per}}(S, h)$ where $h:=\gamma^2/2+1$. Note that by the choice of $\overline{\beta}$ one can find a smooth deformation $S'$ of $S$ in which $\gamma$ becomes irreducible, and hence a perverse stable pair on $S'$ in this class is the same as a usual stable pair. From this and Proposition \ref{perv-smoothness}, one can see that $\pP_n^{\rm{per}}(S, h)$ is deformation equivalent to $\pP_n(\mathsf{K3}, h)$.
  We  have the analogs of Propositions \ref{loc-free}, 
 \ref{smoothness}, \ref{prop:typeI}
and Corollary \ref{decompos} that can be summarized as follows:

\begin{prop}\label{perv-smoothness}
The type I components and isolated type II components of the moduli space of perverse stable pairs $\pP^{\rm{per}}(\tX,\overline{\beta})$ are smooth. 
\end{prop}
\begin{proof}

We first prove the proposition for isolated type II components. The proof is along the lines of the proof of \cite[Proposition C.2]{a44}. We show that  the obstruction space vanishes for any given perverse stable pair $I^\bullet_{\text{per}}=(\O_{\tX} \xrightarrow{s} F)$ with $F=i_*G$ is supported on a $K3$ fiber $i:S\hookrightarrow \tX$. 
We will slightly abuse the notation and use the same symbol $I^\bullet_{\text{per}}$ for the perverse stable pair on $S$ and its pushforward on $\tX$.  As in the proof of [ibid], the obstruction space is the kernel of $v$ in the following natural exact sequence: 
\begin{align*}\Hom(I^\bullet_{\text{per}},G)&\to \Ext^1(G,G) \xrightarrow{u} \Ext^1(\O,G)\\ &\to \Ext^1(I^\bullet_{\text{per}},G) \xrightarrow{v} \Ext^2(G,G)_0 \to 0.\end{align*}
Lemma \ref{stable-perv}, and Serre duality implies that $\Ext^2(G,G)_0 = 0$, so to prove the claim we need to show that the map $u$ is surjective.
Let $S_0:=h(S)$, $p:S_0\to \text{Spec}\,\CC$ be the structure morphism,  and $$G_0:=\dR h_*G \in \Coh(S_0), \quad D_0:=\text{Supp}(G_0).$$ We can write $\dR\Gamma(G)=\dR p_*G_0$ and hence $$\Ext^1(\O,G)\cong H^1(G_0).$$
On the other hand, applying $\dR h_*$ to $\dR\hH om (G,G)\xrightarrow{\text{tr}} \O_{S}\xrightarrow{s} G $, we get the chain of maps $$\dR h_*\dR\hH om (G,G)\to \dR\hH om(G_0,G_0) \xrightarrow{\text{tr}} \O_{S_0}\xrightarrow{s_0} G_0. $$ Clearly, $s_0$ has a zero dimensional cokernel, and the trace map $$\hH om(G_0,G_0) \to \O_{D_0}$$ is surjective with at most a zero dimensional kernel, and therefore $s_0 \circ \text{tr}$ induces the surjection $H^1(\hH om(G_0,G_0))\to H^1(G_0)$. By the local to global spectral sequence

$$H^1(\hH^0(\dR h_* \dR\hH om (G,G))) \subseteq \Ext^1(G,G), $$ and hence $u$ is surjective as desired.

The proof of proposition for the type I components essentially follows the same analysis as above but this time carried out relative to the base curve $C$. 
\end{proof}

\begin{prop}\label{perv-cycles}
Suppose that any component of $\pP^{\rm{per}}_{\rm{II}}$ is isolated. Then we have
\begin{align*} 
&[\pP^{\rm{per}}_{\rm{II}},E^{\bullet}_{\rm{per}}]^{\rm{vir}}=[\pP^{\rm{per}}_{\rm{II}}]\cap c_{top}(\Omega_{\pP^{\rm{per}}_{\rm{II}}}),\\
&[\pP^{\rm{per}}_{\rm{I}},E^{\bullet}_{\rm{per}}]^{\rm{vir}}=[\pP^{\rm{per}}_{\rm{I}}]\cap \big(c_1(\rho^{*}\widetilde{\mathcal{K}}^{\vee})\cup c_{top}(\Omega_{\pP^{\rm{per}}_{\rm{I}}/\tC})\big), 
\end{align*}
where $\widetilde{\mathcal{K}}=\widetilde{\pi}_*\omega_{\tX/\tC}$. 

\end{prop}

\begin{proof}
The same as the proofs of Proposition \ref{loc-free} and  Corollary \ref{decompos} in which the stable pair $I^\bullet$ is replaced with the perverse stable pair $I^\bullet_{\text{per}}=(\O_{\tX} \xrightarrow{s} F)$. The surjectivity $\Ext^1(\F,\F) \to \Ext^1(\O_{\tX},F) =H^1(F)$ which is one of the important requirements can be obtained as the proof of surjectivity of the map $u$ in Proposition \ref{perv-smoothness}.  
\end{proof}

Now we provide an analog of Theorem \ref{thm:main formula} for perverse stable pair invariants:
\begin{thm} \label{thm:main formula-perv}
For the smooth $K3$ fibration $\tpi:\tX\to \tC$, 
suppose that $\overline{\beta}$ 
is given as in \eqref{equ:classs2}, and assume Conjecture \ref{conj:support}.
Then we have 
\begin{align*} 
\PT^{\rm{per}}(\tX)_{\overline{\beta}}&=\sum_{h=0}^{\infty}\sum_{n=1-h}^{\infty}(-1)^{n-1}\chi(\pP^{\rm{per}}_{n}(S,h))\cdot NL^{\tpi}_{h, \overline{\beta}}
q^n,
\end{align*} where $S$ is a $K3$ fiber of $\tX$. \hfill\qed
\end{thm}
\begin{rmk}
By the discussion before Proposition \ref{perv-smoothness}, we can see that the Euler characteristics $\chi(\pP^{\rm{per}}_{n}(S,h))=\chi(\pP_{n}(\mathsf{K3},h))$ in Theorem \ref{thm:main formula-perv} can be read off from Kawai-Yoshioka's formula \eqref{equ:KW}. 
\end{rmk}

\begin{rmk}
If the conjectural relation (\ref{PT=perv2}) holds, 
then by Theorem~\ref{thm:main formula-perv}, it enables
us to compute stable pair invariants
for the nodal $K3$ fibration $X \to C$
with irreducible curve classes. 
\end{rmk}

\section{Wall-crossing formula for $K3$ fibrations} \label{sec:wall}
In this section, we study the stable pair invariants on 
$K3$ fibrations from a different approach. 
We first relate the stable pair invariants 
on a $K3$ fibration (with possibly singular fibers)
to the generalized DT invariants counting 
torsion sheaves supported on fibers of 
the $K3$ fibration. This is achieved by an argument similar to~\cite{TodK3}, 
where a similar problem was studied 
for the trivial $K3$ fibration 
$S \times \mathbb{C} \to \mathbb{C}$
for a $K3$ surface $S$.  
The key ingredient was the 
wall-crossing formula of DT type invariants 
in the derived category~\cite{JS}, \cite{K-S}. 
The latter invariants are calculated in \cite{G-S}
for nodal $K3$ fibrations, and  
we apply its result to calculate the 
stable pair invariants on nodal $K3$ fibrations
with irreducible curve classes. 
For a non-irreducible curve classes, the two approaches 
give different descriptions of the stable pair invariants on 
nodal $K3$ fibrations. 
This gives a non-trivial relationship between the 
generalized DT invariants on nodal $K3$ fibrations and the 
stable pair invariants on smooth $K3$ fibrations, 
which may have some applications to 
the study of the generalized DT invariants.   
\subsection{Stability conditions for $K3$ fibrations}
Let $X$ 
be a smooth 
projective Calabi-Yau 3-fold over $\mathbb{C}$, i.e. 
\begin{align*}
K_X =0, \quad H^1(X, \oO_X)=0. 
\end{align*}
We assume that there is a 
morphism 
\begin{align*}
\pi \colon X \to \mathbb{P}^1
\end{align*}
such that every scheme-theoretic fiber is an integral scheme. 
Note that a generic fiber
of $\pi$ is a smooth algebraic $K3$ surface. 
An example is given in Example~\ref{exam:nodal}. 
Let $\Coh_{\pi}(X)$ be the subcategory of 
$E \in \Coh(X)$ such that $\Supp(E)$ is contained in 
the fibers of $\pi$. 
We consider 
its bounded derived category
\begin{align*}
\dD_0 \cneq D^b \Coh_{\pi}(X). 
\end{align*}
Note that for any object $E \in \dD_0$, we have 
\begin{align}\notag
\ch(E) \sqrt{\td_X} &=(0, r[F], \beta, n) \\
\label{stab:ch}&\in H^0(X) \oplus H^2(X)
 \oplus H^4(X) \oplus H^6(X). 
\end{align}
Here $[F]$ is a fiber class of $\pi$, 
and $r, n \in \mathbb{Z}$. 
Under the Poincar\'e duality, the element $\beta$
is regarded as an element in $H_2(X, \mathbb{Z})^{\pi}$. 
By setting 
\begin{align*}
\Gamma_0 \cneq \mathbb{Z} \oplus H_2(X, \mathbb{Z})^{\pi} \oplus \mathbb{Z}
\end{align*}
the expression (\ref{stab:ch}) gives a group homomorphism
\begin{align}\label{map:cl0}
\cl_0 \colon K(\dD_0) \to \Gamma_0, \quad 
E \mapsto (r, \beta, n). 
\end{align}
Let $\omega$ be an ample divisor 
on $X$. 
Using the map (\ref{map:cl0}), we 
define the slope function $\mu_{\omega}$ on 
$\Coh_{\pi}(X)$ by 
\begin{align*}
\mu_{\omega}(E)= \int_{X} \omega \cdot \beta / r \in \mathbb{Q} 
\cup \{\infty\}.
\end{align*}
Here $\cl_0(E)=(r, \beta, n)$ and 
$\mu_{\omega}(E)=\infty$ if $r=0$, i.e. 
$E$ is one or zero dimensional 
sheaf. 
The slope function $\mu_{\omega}$ defines the 
$\mu_{\omega}$-stability on $\Coh_{\pi}(X)$ in
an obvious way. 
We define the following full subcategories in $\Coh_{\pi}(X)$:
\begin{align*}
\tT_{\omega} &\cneq \langle F : \mu_{\omega} \mbox{-semistable with }
\mu_{\omega}(F)>0 \rangle_{\rm{ex}} \\
\fF_{\omega} &\cneq \langle F : \mu_{\omega} \mbox{-semistable with }
\mu_{\omega}(F) \le 0 \rangle_{\rm{ex}}. 
\end{align*}
Here $\langle \ast \rangle_{\rm{ex}}$ is the 
extension closure of $\ast$, i.e. the smallest
extension closed subcategory which contains $\ast$.  
The above pair $(\tT_{\omega}, \fF_{\omega})$
determines a torsion pair (cf.~\cite{HRS}) on $\Coh_{\pi}(X)$.
We have the associated tilting
\begin{align*}
\bB_{\omega} \cneq \langle \fF_{\omega}, \tT_{\omega}[-1] \rangle_{\rm{ex}}
\subset D^b \Coh_{\pi}(X). 
\end{align*}
The subcategory $\bB_{\omega}$ is the heart 
of a bounded t-structure of $D^b \Coh_{\pi}(X)$, 
hence in particular it is an abelian category. 
Let $Z_{\omega, 0} \colon \Gamma_0 \to \mathbb{C}$
be the group homomorphism defined by 
\begin{align}\notag
Z_{\omega, 0}(v) &\cneq \int_X e^{-i\omega} v \\
\label{hom:Z}
&=n- \frac{r}{2} (\omega|_{X_p})^2 -(\omega \cdot \beta) \sqrt{-1}. 
\end{align}
Here we have written $v=(r, \beta, n)$, 
and $X_p=\pi^{-1}(p)$ for a closed point $p\in C$. 
Below we assume the 
familiarity of Bridgeland stability conditions~\cite{Brs1}
on triangulated categories. 
\begin{lem}\label{pair:B}
The pairs
\begin{align}\notag
\sigma_{t\omega, 0} \cneq 
(Z_{t\omega, 0}, \bB_{\omega}), \quad t>0
\end{align}
determine Bridgeland stability conditions on $\dD_0$. 
\end{lem}
\begin{proof}
The proof is almost the same as in the $K3$ surface case 
in~\cite[Proposition~7.1]{Brs2}. 
However, we need to take a little care since there may
be possible singular fibers of $\pi$. 
For non-zero $E \in \bB_{\omega}$, let us show the following property
\begin{align}\label{positivity}
Z_{t\omega, 0}(E) \in \{ r \exp(i\pi \phi) : r>0, \phi \in (0, 1] \}. 
\end{align}
By our construction of $\bB_{\omega}$,
we have $\Imm Z_{t\omega, 0}(E)\ge 0$. 
If $\Imm Z_{t\omega, 0}(E)=0$, then 
$\hH^1(E)$ is zero dimensional and 
$\hH^0(E)$ is a 
$\mu_{\omega}$-semistable sheaf 
with $\mu_{\omega}(\hH^0(E))=0$. 
Therefore by (\ref{hom:Z}), 
it is enough to show that 
for any $\mu_{\omega}$-stable sheaf 
$F \in \Coh_{\pi}(X)$
with $\cl_0(F)=(r, \beta, n)$
and $\omega \cdot \beta=0$, we have 
$n\le 0$. 
Since $F$ is $\mu_{\omega}$-stable, it is 
an $\oO_{X_p}$-module 
for some $p \in \mathbb{P}^1$. 
If $X_p$ is smooth, then
the inequality
$n\le 0$ this is a 
consequence of Bogomolov-Gieseker inequality. 
In general, since $X_p$ is an integral scheme, 
the structure sheaf $\oO_{X_p}$ is a $\mu_{\omega}$-stable sheaf
with $\mu_{\omega}(\oO_{X_p})=0$. 
By the $\mu_{\omega}$-stability of $F$, if 
$r\ge 2$, we have $\Hom(\oO_{X_p}, F)=\Hom(F, \oO_{X_p})=0$. 
By the Serre duality and the Riemann-Roch theorem, we obtain
\begin{align}\label{RR}
\chi(F) = r+n \le 0. 
\end{align}
If $r=1$, 
we may assume that
$\Hom(F, \oO_{X_p}) =0$. 
Indeed if otherwise, then $F$ is a ideal sheaf of 
$\oO_{X_p}$, hence the inequality $n\le 0$. 
Then we also have $\Hom(\oO_{X_p}, F)=0$. 
Indeed if there is a non-zero morphism 
$\oO_{X_p} \to F$, then it must be injective and the 
cokernel must be a zero dimensional sheaf. 
But as $\Ext_{X_p}^1(\oO_x, \oO_{X_p})=0$ for any $x\in X_p$, 
the sheaf $F$ must contain a zero dimensional sheaf, 
which contradicts to the $\mu_{\omega}$-stability of $F$. 
Therefore $\Hom(F, \oO_{X_p})=\Hom(\oO_{X_p}, F)=0$, 
and the Riemann-Roch computation (\ref{RR}) shows $n\le 0$. 

We finally check the Harder-Narasimhan property. 
By following the arguments 
as in~\cite[Proposition~7.1]{Brs2}, 
we finally arrive at the exact sequences of sheaves
for $i\ge 1$
\begin{align*}
0 \to Q \to \hH^{-1}(E_i) \to \hH^0(L_i) \to 0. 
\end{align*}
Here we have used the same notation of~\cite[Proposition~7.1]{Brs2}. 
The sheaf $Q$ is a pure two dimensional sheaf and $\hH^0(L_i)$ is a 
zero dimensional sheaf. 
A difference from~\cite[Proposition~7.1]{Brs2} is that our 
sheaf $Q$ is no longer a torsion free sheaf. 
However by pushing forward the above sequence by the 
generic projection $\Supp(\qQ) \to \mathbb{P}^2$, we 
arrive at the same situation in~\cite[Proposition~7.1]{Brs2}. 
Therefore the Harder-Narasimhan property also holds. 
\end{proof}

\subsection{Weak stability conditions on $\dD$}
Let $\dD$ be the triangulated subcategory of $D^b \Coh(X)$
defined by 
\begin{align*}
\dD \cneq \langle \pi^{\ast} \Pic(\mathbb{P}^1), 
\Coh_{\pi}(X)  \rangle_{\rm{tr}} \subset D^b \Coh(X).
\end{align*}
Here $\langle \ast \rangle_{\rm{tr}}$ is the
triangulated closure, i.e. 
the smallest triangulated subcategory which contains $\ast$. 
We construct weak stability conditions on $\dD$ in the 
sense of~\cite{Tcurve1}, following the same
argument of~\cite[Subsection~3.3]{TodK3}. 
We refer to~\cite[Section~2]{Tcurve1} for details on the space of 
weak stability conditions on triangulated categories.
Let $\aA_{\omega}$ be the subcategory 
of $\dD$ defined by 
\begin{align*}
\aA_{\omega} \cneq \langle \pi^{\ast}\Pic(\mathbb{P}^1), \bB_{\omega}
\rangle_{\rm{ex}} \subset \dD. 
\end{align*}
One can check that, using the same argument of~\cite[Proposition~2.9]{TodK3}, 
the subcategory $\aA_{\omega}$ is the heart of a bounded t-structure on $\dD$. 
In particular it is an abelian category. 
We define the group homomorphism
\begin{align*}
\cl \colon K(\dD) \to \Gamma \cneq \mathbb{Z} \oplus \Gamma_0
\end{align*}
by $\cl|_{K(\dD_0)}=(0, \cl_0)$ and 
\begin{align*}
\cl(\pi^{\ast}\lL)= (1, \deg \lL, 0, 0)
\end{align*}
for $\lL \in \Pic(\mathbb{P}^1)$. 
We take a filtration $\Gamma_{\bullet}$
of $\Gamma$ by
\begin{align*}
0= \Gamma_{-1} \subset \Gamma_0 \subset \Gamma_1 \cneq \Gamma
\end{align*}
where the second inclusion is given by $v \mapsto (0, v)$. 
We define the element 
\begin{align*}
Z_{\omega} \in \prod_{i=0}^{1} \Hom(\Gamma_i/\Gamma_{i-1}, \mathbb{C})
\end{align*}
to be the following:
\begin{align*}
&Z_{\omega, 1}(R) \cneq R \sqrt{-1}, 
\quad R \in \Gamma_1/\Gamma_0 =\mathbb{Z} \\
&Z_{\omega, 0}(v) \cneq \int_X e^{-i\omega}v, \quad v\in \Gamma_0. 
\end{align*}
\begin{lem}
The pairs 
\begin{align*}
\sigma_{t\omega} \cneq (Z_{t\omega}, \aA_{\omega}), \quad 
t\in \mathbb{R}_{>0}
\end{align*}
determine weak stability conditions on $\dD$
with respect to the filtration $\Gamma_{\bullet}$
on $\Gamma$. 
\end{lem}
\begin{proof}
Using Lemma~\ref{pair:B} instead of~\cite[Lemma~3.3]{TodK3}, 
the same argument of~\cite[Lemma~3.4]{TodK3} is applied
without any major modification. 
\end{proof}

\subsection{Donaldson-Thomas type invariants}
Let $M$ be a $\mathbb{C}$-scheme. 
In the paper~\cite{a1}, Behrend
associated a canonical constructible function $\nu$ on it, 
called \textit{Behrend function}. 
The $\nu$-weighted Euler characteristic
\begin{align}\label{DTtype}
\int_{M} \nu \ d\chi \cneq \sum_{m\in \mathbb{Z}}
m \cdot \chi(\nu^{-1}(m))
\end{align}
is called the DT-type invariants. 
This is due to the fact by
Behrend~\cite{a1} that, if 
$M$ is proper and admits a 
symmetric perfect obstruction theory, 
we have the identity
\begin{align*}
\int_{[M]^{\rm{\rm{vir}}}}1 = \int_{M} \nu \ d\chi. 
\end{align*}
On the other hand, the 
naive Euler characteristic $\chi(M)$ is 
called the Euler characteristic version
of the invariant (\ref{DTtype}). 

Since we assumed that $X$ is a smooth projective Calabi-Yau 3-fold, 
the perfect obstruction theory in Theorem~\ref{PT}
is symmetric, hence we have 
\begin{align*}
P_{n, \beta}=
\int_{\pP_n(X, \beta)} \nu \ d\chi. 
\end{align*}
We can similarly define DT type 
invariants on $X$ using the Behrend function. 
For $(r, \beta, n) \in \Gamma_0$, let 
$\mM_{\omega}(r, \beta, n)$ be the 
moduli stack of $\omega$-Gieseker stable 
sheaves $E\in \Coh_{\pi}(X)$ with 
$\cl_0(E)=(r, \beta, n)$. 
If any closed point of $\mM_{\omega}(r, \beta, n)$
is stable, then it is a 
$\mathbb{C}^{\ast}$-gerbe over a 
projective scheme, hence the 
following integration makes sense:
\begin{align}\label{inv:J}
J(r, \beta, n) 
=\int_{\mM_{\omega}(r, \beta, n)} \nu \ d\chi. 
\end{align} 
In general, 
there may be a 
closed point of 
$\mM_{\omega}(r, \beta, n)$ 
corresponding to a strictly semistable sheaf. 
In such a case, the invariant (\ref{inv:J})
is defined in~\cite{JS}, \cite{K-S}, 
as the generalized DT invariant, 
and takes its value in $\mathbb{Q}$. 
It
is defined as
 the weighted Euler 
number of the `logarithm' of 
$\mM_{\omega}(r, \beta, n)$ in the 
motivic Hall algebra
of $\Coh_{\pi}(X)$. 
An Euler characteristic version of the 
invariant (\ref{inv:J}) for the 
trivial $K3$ fibration 
$S \times \mathbb{C} \to \mathbb{C}$
for a $K3$ surface $S$
is available in~\cite[Definition~4.23]{TodK3}. 
\begin{lem}
The invariant (\ref{inv:J})
is independent of a choice of $\omega$. 
\end{lem}
\begin{proof}
Since the Euler pairing on $\dD_0$ is trivial, the 
same argument of~\cite[Lemma~4.16]{TodK3}
is applied. 
\end{proof}

\subsection{Wall-crossing formula}
The following is the main result in this section: 
\begin{thm}\label{thm:WCF}
We have the following formula
\begin{align}\notag
\PT(X)=\prod_{r\ge 0, \beta>0, n\ge 0} &
\exp\left((-1)^{n-1}J(r, \beta, r+n)q^n t^{\beta} \right)^{n+2r} \\
\label{main:formula}
&\cdot \prod_{r> 0, \beta>0, n>0}
\exp\left((-1)^{n-1}J(r, \beta, r+n) q^{-n}t^{\beta} \right)^{n+2r}.
\end{align}
\end{thm}
\begin{proof}
The Euler characteristic version of the 
above result is proved when $X$ is a trivial fibration in~\cite{TodK3}. 
The proof proceeds along with the same argument of~\cite{TodK3}
without any major modification, 
using 
the derived category version of~\cite[Theorem~5.18]{JS}
proved in~\cite[Theorem~2.8]{TodHall}
instead of~\cite[Theorem~6.28]{Joy4}. 
We only give an outline of the proof. 

\begin{step}
First wall-crossing. 
\end{step}
By the result of~\cite{BrH}, \cite{TodHall}
(also see~\cite[Theorem~1.3]{Tolim2},~\cite[Theorem~3.11]{Tsurvey}),  
we have the following formula
\begin{align}\label{PT:first}
&\PT(X) = \\ \notag
&\prod_{\beta>0, n>0} 
\exp \left((-1)^{n-1}J(0, \beta, n)  q^n t^{\beta}\right)^n
\left( \sum_{\beta, n} L(\beta, n) q^n t^{\beta} \right). 
\end{align}
The above formula is obtained by 
applying the wall-crossing formula of 
DT type invariants in the 
category of perverse coherent sheaves on $D^b \Coh(X)$, 
which are certain two term complexes of coherent sheaves. 
The invariant $L(\beta, n)$ counts perverse coherent sheaves 
$E\in D^b \Coh(X)$,
 which are semistable with respect to a certain
self dual weak 
stability condition on the perverse heart.
(It was denoted by $\mu_{i\omega}$-limit 
stability in~\cite[Section~3]{Tolim2}, and $Z_{\omega, 1/2}$-stability
in~\cite[Section~5]{Tsurvey}.) 
The object $E$ satisfies the numerical 
condition
\begin{align*}
\ch(E)=(1, 0, -\beta, -n) \in H^{\ast}(X, \mathbb{Q}).
\end{align*}
The precise definition of the invariant $L(\beta, n)$
is available in~\cite[Definition~5.5]{Tsurvey}, 
where it is denoted by $L_{n, \beta}$. 
When $X=S\times \mathbb{P}^1$ for a $K3$ surface $S$, 
its Euler characteristic version is 
available in~\cite[Subsection~4.6]{TodK3}.

\begin{step}
Generating series of DT type invariants. 
\end{step}
For $(r, \beta, n) \in \Gamma_0$, let 
$\mM_{t\omega}(r, \beta, n)$
be the moduli stack of $\sigma_{t\omega}$-semistable 
objects $E \in \aA_{\omega}$ satisfying
\begin{align*}
\cl(E)=(1, -r, -\beta, -n) \in \Gamma. 
\end{align*}
Similarly to~\cite[Lemma~4.13]{TodK3}, 
the above moduli stack is 
realized as a constructible subset of 
an Artin stack $\mM$ locally of finite type. 
Similarly to
the previous subsection and~\cite[Definition~4.10]{TodK3}, 
it 
defines the (generalized) DT type invariant
\begin{align*}
\DT_{t\omega}(r, \beta, n)  \cneq 
\int_{\mM_{t\omega}(r, \beta, n)} \nu \ d\chi
 \in \mathbb{Q}. 
\end{align*}
We form the following generating series
\begin{align*}
\DT_{t\omega}(X)
\cneq \sum_{(r, \beta, n) \in \Gamma_0} \DT_{t\omega}(r, \beta, n)
q^n t^{\beta} s^r. 
\end{align*}
We consider the behavior of the above 
generating series for $t\gg 0$ and $0< t\ll 1$. 
We have the following proposition:
\begin{prop}\label{prop:t}
We have the following formulas
\begin{align}
\notag
&\lim_{t\to \infty} \DT_{t\omega}(X)=
\sum_{r, \beta, n} L(\beta, n) q^n t^{\beta} s^r \\
\notag
&\lim_{t\to +0} 
\DT_{t\omega}(X)= \lim_{t\to +0} \sum_{r, \beta} 
\DT_{t\omega}(r, \beta, 0)t^{\beta} s^r. 
\end{align} 
\end{prop}
\begin{proof}
The same proof of~\cite[Proposition~4.16]{TodK3}
is applied without any major modification. 
\end{proof}

\begin{step}
Wall-crossing in $\aA_{\omega}$. 
\end{step}
We investigate the difference of the generating 
series $\DT_{t\omega}^{}(X)$ for 
$ t=t_0+ 0$ and $t=t_0-0$ for a fixed 
$t_0\in \mathbb{R}_{>0}$. 
We have the following result:
\begin{prop}
We have the following formula
\begin{align}\notag
\lim_{t\to t_{0} + 0}
\DT_{t\omega}(X) 
= &\lim_{t\to t_{0} -0} \DT_{t\omega}(X) 
\\
\notag
 &\cdot \prod_{\begin{subarray}{c} \beta>0, \\
n=\frac{1}{2}rt_0^2 \omega|_{X_p}^2
\end{subarray}} 
\exp\left( (-1)^{n-1}J(r, \beta, r+n)  q^n t^{\beta} s^r \right)^{\epsilon(r)(n+2r)}. 
\end{align}
Here $\epsilon(r)=1$ if $r>0$, $\epsilon(r)=-1$ if $r<0$
and $\epsilon(r)=0$ if $r=0$. 
\end{prop}
\begin{proof}
The same proof of~\cite[Theorem~5.1]{TodK3}
is applied by using the results of~\cite{JS}, \cite{TodHall}
 instead of~\cite{Joy4}. 
We just give an outline of the proof. 
For $v \in \Gamma_0$, let 
$\mM_{t\omega, 0}(v)$ be the moduli stack of 
$\sigma_{t\omega, 0}$-semistable objects
$E \in \bB_{\omega}$ satisfying 
$\cl_0(E)=v$. Similarly to the 
invariant $J(v)$, we have the generalized DT type invariant 
\begin{align*}
J_{t\omega}(v) \cneq \int_{\mM_{t\omega, 0}(v)} \nu \ d\chi \in \mathbb{Q}
\end{align*}
which counts $\sigma_{t\omega, 0}$-semistable objects $E\in \bB_{\omega}$
with $\cl_0(E)=v$. 
The wall-crossing formula~\cite{JS} describes the difference of 
the two limiting series $\DT_{t\omega\pm 0}(X)$ in
terms of the invariants
$J_{t\omega}(v)$, 
where $v\in \Gamma_0$ satisfies  
$Z_{t\omega, 0}(v) \in \mathbb{R}_{>0} \sqrt{-1}$.  
If we write $v=(r, \beta, n)$, the 
latter condition implies that
\begin{align*}
n= \frac{1}{2}r t_0^2 \omega|_{X_p}^2.
\end{align*} 
On the other hand, the proof of~\cite[Corollary~4.27]{TodK3} shows that 
the invariant $J_{t\omega}(v)$ is independent of 
$t$, and coincides with $J(v)$. Now by applying the wall-crossing formula 
in~\cite{JS}, and computing the relevant combinatorial coefficients 
as in~\cite{Tolim2},~\cite{Tcurve1}, we obtain the result. 
\end{proof}
By combining the above result with Proposition~\ref{prop:t}, we have the 
following corollary (cf.~\cite[Corollary~5.2]{TodK3}):
\begin{cor}\label{cor:WCF}
We have the following formula: 
\begin{align*}
&\sum_{(r, \beta, n) \in \Gamma_0}L(\beta, n) q^n t^{\beta} s^r \\
&=\prod_{\beta>0, rn>0} 
\exp\left( (-1)^{n-1}J(r, \beta, r+n) q^n t^{\beta} s^r  \right)^{\epsilon(r)(n+2r)} 
\cdot 
\widehat{\DT}(X)
\end{align*}
where $\widehat{\DT}(X)$ is defined to be
\begin{align*}
\widehat{\DT}(X) \cneq 
\lim_{t\to +0} \sum_{(r, \beta, n) \in \Gamma_0}
\DT_{t\omega}(r, \beta, 0)t^{\beta} s^r
\end{align*}
\end{cor}

\begin{step}
Further wall-crossing. 
\end{step}
The final step is to decompose the series
$\widehat{\DT}(X)$
into a product which involves the invariants $J(v)$. 
Following~\cite[Subsection~3.6]{TodK3}, 
let 
\begin{align*}
\aA_{\omega}(1/2) \subset \aA_{\omega}
\end{align*}
be the subcategory 
given as the extension closure of $Z_{t\omega}$-semistable 
objects $E \in \aA_{\omega}$ for $0<t\ll 1$, satisfying 
\begin{align*}
\lim_{t\to +0} \arg Z_{t\omega}(E)= \frac{\pi}{2}.
\end{align*}
Note that any object $E \in \aA_{\omega}$ which 
contributes to 
the invariant $\DT_{t\omega}(r, \beta, 0)$
is an object in $\aA_{\omega}(1/2)$. 
Similarly to~\cite[Subsection~3.7]{TodK3}, 
we are able to construct a one parameter family of 
weak stability conditions 
\begin{align*}
(\widehat{Z}_{\omega, \theta}, \aA_{\omega}(1/2)), \ 
\theta \in (0, 1)
\end{align*}
on $\aA_{\omega}(1/2)$.
Moreover by~\cite[Proposition~3.17]{TodK3},  
an object $E \in \aA_{\omega}(1/2)$ contributes to 
the invariant $\DT_{t\omega}(r, \beta, 0)$
if and only if it is $\widehat{Z}_{\omega, 1/2}$-semistable, 
and an object $E \in \aA_{\omega}(1/2)$ with $\rank(E)=1$
is 
$\widehat{Z}_{t\omega, \theta}$-semistable 
for $0<\theta \ll 1$ if and only if 
it is an object in $\pi^{\ast} \Pic(\mathbb{P}^1)$. 
Applying the wall-crossing formula, we obtain
\begin{align*}
\widehat{\DT}(X) =
\prod_{r>0, \beta>0}
\exp \left(- J(r, \beta, r) t^{\beta} s^{r} \right)^{2r}
\cdot \sum_{r\in \mathbb{Z}} s^r. 
\end{align*}
The above formula is obtained by the same 
argument of~\cite[Proposition~5.3]{TodK3}
without any major modification. 
Combined with (\ref{PT:first}), Corollary~\ref{cor:WCF}, 
and taking the $s^0$-term, 
we obtain the desired identity (\ref{main:formula}). 
\end{proof}

\subsection{Irreducible curve class case}
When 
the curve class $\beta$ is irreducible, 
we have the following corollary of 
Theorem~\ref{thm:WCF}: 
\begin{cor}
Suppose that $\beta \in H_2(X, \mathbb{Z})^{\pi}$ is 
irreducible. Then for $n\ge 0$,  we have 
\begin{align}\label{d2-d4}
P_{n, \beta} &=\sum_{r\ge 0} (-1)^{n-1} (n+2r) J(r, \beta, r+n) \notag\\
P_{-n, \beta} &=\sum_{r>0} (-1)^{n-1}(n+2r) J(r, \beta, r+n). 
\end{align}
\end{cor}
\begin{rmk}
Note that when the data $(r,\beta,n)$ is chosen so that there does not exist any strictly semistable sheaf, then the invariants $J(r,\beta,n)$ appearing on the right hand side of \eqref{d2-d4} coincide with the invariants computed in \cite[Section 2]{G-S}. 
\end{rmk}

\bibliographystyle{amsalpha}
\bibliography{ref1}

\noindent {\tt{amingh@math.umd.edu}} \\
\noindent {\tt{University of Maryland}} \\
\noindent {\tt{College Park, MD 20742-4015, USA}} \\

\noindent {\tt{artan@mit.edu}} \\
\noindent {\tt{Massachusetts Institute of Technology (MIT), Department of Mathematics}} \\
\noindent {\tt{Building 2, Room 2-247, 77 Massachusetts Avenue, Cambridge, MA, USA 02139-4307}} \\

\noindent {\tt{yukinobu.toda@ipmu.jp}} \\
\noindent {\tt{Institute for the Physics and Mathematics of the Universe}} \\
\noindent {\tt{University of Tokyo, 5-1-5 Kashiwanoha, Kashiwa, 277-8583, Japan}} \\

\end{document}